\documentclass{article}

\usepackage{amsthm}
\usepackage{xypic}
\usepackage{amssymb,amsfonts,amsmath}
\usepackage{amscd}
\usepackage[english]{babel}
\usepackage{import}
\usepackage{mathtools}
\usepackage[colorlinks=true]{hyperref}
\usepackage{subfiles}
\usepackage{comment}

\newtheorem{theorem}{Theorem}[subsection]
\newtheorem{corollary}[theorem]{Corollary}
\newtheorem{proposition}[theorem]{Proposition}

\theoremstyle{definition}
\newtheorem{definition}[theorem]{Definition}

\theoremstyle{plain}
\newtheorem{lemma}[theorem]{Lemma}
\theoremstyle{remark}
\newtheorem{remark}[theorem]{Remark}

\newtheorem{warning}[theorem]{Warning}

\def \quo{\backslash}

\newcommand{\m}{\mathcal}

\def\Epsilon{\mathcal{E}}
\def\AA{\mathbb{A}}

\def\ZZ{\mathbb{Z}}

\def\Z2{\ZZ/2}

\def\({\left(}
\def\){\right)}

\DeclareMathOperator\Fun{Fun}
\DeclareMathOperator\RFun{RFun}

\DeclareMathOperator\Hom{Hom}
\DeclareMathOperator\Ext{Ext}

\DeclareMathOperator\pro{Pro}
\DeclareMathOperator\ind{Ind}
\DeclareMathOperator\HH{H}

\newcommand\Cref[1]{{Corollary~\ref{#1}}}

\newcommand\adj{\rightleftarrows}
\newcommand\Etale{\text{\'Etale }}
\newcommand\etale{\text{\'etale }}
\def\spec{\mathop{\rm Spec}}

\newcommand{\Sh}[1]{Shv_\infty(#1)}
\newcommand{\Prosh}[1]{\pro Shv_\infty(#1)}
\newcommand{\ProSh}[1]{\pro Shv_\infty(#1)}

\newcommand\Et{\text{\'Et}}
\newcommand\et{\text{\'et}}

\newcommand{\prl}{\mathfrak{Pr}^L}

\newcommand{\mt}{\m{T}}
\newcommand{\mtl}{\Sh{\m{T},\lambda}}
\makeatother
\newcommand\IIarr{\ar@<-.4ex>[r] \ar@<.4ex>[r]}
\newcommand\IIard{\ar@<-.4ex>[d] \ar@<.4ex>[d]}
\newcommand{\IIIarr}{\ar@<-.8ex>[r] \ar[r] \ar@<.8ex>[r]}
\newcommand{\IIIard}{\ar@<-.8ex>[d] \ar[d] \ar@<.8ex>[d]}
\newcommand{\IVarr}{\ar@<-1.2ex>[r] \ar@<-.4ex>[r] \ar@<.4ex>[r] \ar@<1.2ex>[r]}
\newcommand{\IVard}{\ar@<-1.2ex>[d] \ar@<-.4ex>[d] \ar@<.4ex>[d] \ar@<1.2ex>[d]}
\newcommand{\invlim}{\varprojlim}
\newcommand{\colim}{\varinjlim}
\newcommand{\laxlim}{\underleftarrow{\operatorname{Lax}}}
\newcommand{\oplaxlim}{\underleftarrow{\operatorname{OpLax}}}
\newcommand{\into}{\hookrightarrow}

\title{\Etale Homotopy Obstructions of Arithmetic Spheres}
\author{Edo Arad, Shachar Carmeli, Tomer M. Schlank}
\date{March 2017}
\begin{document}

\maketitle

\begin{abstract}
Let $K$ be a field of characteristic $\ne 2$ and let $X$ be the affine variety 
over $K$ defined by the equation
\[
X:\ a_0x_0^2 + \cdots + a_nx_n^2 = 1
\]
where $n\ge 0$ and $a_i\in K$.
In this paper we compute the lowest mod 2 \'etale homological obstruction class to the 
existence of a $K$-rational point on $X$, and show that it is the cup product of the form 
\[ o_{n+1} = [a_0]\cup\cdots\cup[a_n]. \]

Our computation is an \Etale-homotopy analogue 
of the topological fact that Stiefel-Whitney classes are the homological obstructions to find a section to the unit sphere bundle of a real vector bundle.   
\end{abstract}

\tableofcontents
\section{Introduction}
\label{intro}
The arithmetic of quadratic forms is a well-established subject. 
In 1987 Hilbert introduced the quadratic symbol over a local field $K$, 
defined by
\[ (a,b) = 
\begin{cases}
1 & \text{ if } ax^2 + by^2 = 1 \text{ is solvable over } K \\  
-1 & \text{ else }
\end{cases}
\]
and its generalization to number fields using the various completions at 
the places of the field.

Later on, those symbols recognized as a special case of a general Hasse-Witt invariant 
$HW_2(B) \in Br(K)$, attached to a quadratic form $B$ over a field $K$ of characteristic 
$\ne 2$. Those classes are arithmetic analogues of the Stiefel-Whitney classes 
in algebraic topology, and form part of a sequence $HW_k(B) \in \HH^k_\et(K,\mu_2)$ of classes, satisfying the 
Whitney product formula analogous to the product formula for Stiefel-Whitney classes from classical topology. 

Applying the product formula to the top Hasse-Witt invariant, one immediately see that 
$HW_{rank(B)}(B)$ is an obstruction to a solution of the equation 
$B(v,v) = 1$ over $K$, just like the original Hilbert symbol is defined as an obstruction to 
solution of such equation for $rank(B) = 2$. 

In \cite{HarpazSchlank} (and expended in \cite{BARNEA2016784}), 
Y.Harpaz and the third author introduced a general obstruction theory for rational points 
on algebraic varieties based on the notion of relative \etale homotopy type. 

In this paper, we shall link the classical theory of Hasse-Witt classes and the quadratic 
symbol with the general \'etale obstruction theory. 
Specifically, for a quadratic form $B$, we show that the mod-2 \emph{homological} obstruction for 
a solution of the equation $B(v,v) = 1$ 
coincides with the top Hasse-Witt class of $B$. 

\subsection{Topological Motivation and Outline of the Proof}

We present below the outline of the proof and the topological origin of the argument.

\subsubsection{Obstructions for Unit Sections}
Classically, given an $n$-dimensional topological vector bundle $p\colon E\to Y$,
there is an associated cohomology class $SW_n(p)\in \HH^n(Y;\Z2)$, the Stiefel-Whitney class,
which serves as an obstruction to the existence of a global non zero section for $p$.
Vector bundles are classified by maps to the classifying space 
\[
Gr(n,\infty) \cong BGL_n \cong BO_n,
\]
by pulling back the universal vector bundle $\widetilde{p}\colon\widetilde{E}\to BO_n$.
The Stiefel-Whitney class of the vector bundle associated with $f\colon Y \to BO_n$ can be computed as the pullback
\[
SW_n(f^*(\widetilde{p})) = f^*(SW_n(\widetilde{p})).
\]

In the arithmetic setting, we cannot use the topological space $BO_n$ so we consider the analogous stack.
Specifically, given a field $K$ of characteristic different then $2$ we denote by $BO_{n,K}$ the stack  classifying \textbf{algebraic} vector bundles over $K$-schemes equipped with a non-degenerate quadratic form. 
We thus have a universal sphere bundle 
\[
\widetilde{S}\hookrightarrow\widetilde{E}\xrightarrow{\widetilde{p}} BO_{n,K}
\]
which is the variety of all vectors of norm 1 with respect to the universal quadratic form.
Analogously to the topological case, the \etale $\Z2$-cohomology ring of $BO_{n,K}$ is a polynomial ring freely generated by classes
$HW_1,\ldots,HW_n$ of degrees $\deg(HW_i)=i$ over  $H^*_{\acute{e}t}(K,\Z2)$. 

In this paper we prove that given a quadratic bundle $p\colon E\to Y$ classified by $f\colon Y\to BO_{n,K}$,
the class $SW(p)=f^*(HW_n)$ is an obstruction to the existence of a section for the sphere bundle 
$f^*(\widetilde{S})\to Y$.
Furthermore, this is precisely the obstruction class as defined in \cite{BARNEA2016784}.

\subsubsection{The Computation of the Obstruction for the Sphere}
We study the variety $X\colon \sum_{i=0}^n a_i x_i^2 =1$, which is properly 
thought of as the unit sphere in an $n+1$ dimensional linear space equipped with
the quadratic form defined by $\sum_{i=0}^n a_i x_i^2$.
Hence, it corresponds to a (rational) point on $BO_{n+1}$
so that there is a pullback diagram
\[
\xymatrix{
	X   \ar[r]\ar_p[d]    &   \widetilde{S}_{n+1}   \ar^{\widetilde{p}}[d]  \\
	\spec K \ar^f[r]      &   BO_{n+1}.
}
\]

The sphere $X$ is the unit sphere in a quadratic space which can be decomposed as a direct product
of 1 dimensional quadratic spaces of the form $a_ix^2$.
This means we have a commutative diagram
\[
\xymatrix{
	X   \ar[r]\ar_p[d]          & \widetilde{S}_1^{n+1}  \ar[r]\ar[d] &\widetilde{S}_{n+1}   \ar^{\widetilde{p}}[d]  \\
	\spec K \ar^{\times f_i}[r]& BO_1^{n+1} \ar[r] &   BO_{n+1}.
}
\]
The obstruction class $SW(p)=f^*(HW_n)$ can now be computed using Whitney's product formula
\[
SW(p) = f_0^*(HW_1)\cup\cdots\cup f_n^*(HW_1).
\]

In the 0-dimensional case, $X\colon ax^2=1$ 
and the obstruction class is \[[a]\in \HH^1(K;\Z2) \cong K^\times/(K^\times)^2 \]

In general, we have
\[
SW(p) = [a_0]\cup\cdots\cup[a_n].
\]
\subsection{Organization of the Paper}

In Section \ref{EtTopType} we introduce relative \etale homotopy type in the case of
a morphism of $\infty$-stacks. We also prove compatibilities of the relative homotopy and homology types of such morphisms, and derive basic properties of those, such as smooth base-change and behavior with repsect to colimits.
In Section \ref{RelativeObstruction} obstruction theory is introduced in the context of $\infty$-topoi. 
These two sections serve as the theoretical foundation needed for the computation, 
and fix the required notations. 

Quadratic bundles and their classifying stack are introduced and examined in Section \ref{FormsAndMaps}, and the relative \etale topological type is computed for the universal quadratic bundle over these classifying stacks.
Finally, Section \ref{ObstructionsForSpheres} contains the computation of the obstruction class for arithmetic spheres.
\subsection{Acknowledgements}

Shachar Carmeli is supported by the Adams Fellowship Program of the Israel Academy of Sciences and Humanities. Tomer Schlank is supported by the Alon Fellowship and ISF1588/18.

\section{\'Etale Homotopy Type}
\label{EtTopType}
 In this section we shall recall the definition of the 
relative homotopy type as defined in \cite{BARNEA2016784}.
We use the formalism of $\infty$-stacks and $\infty$-topoi
and the machinery developed there. 

We work with the $\infty$-category  $\mathfrak{Top}_\infty$ of $\infty$-topoi, as in \cite[Section 6]{LurieHTT}.

Let $K$ field of characteristic $\ne 2$, fixed throughout the paper
and let $Sch_{/K}$ denote the essentially small site of schemes of finite type over $K$ endowed with the \etale topology. For a site $C$ we denote by $Sh_\infty(C)$ the $\infty$-topos of sheaves of spaces over $C$. 

\begin{definition}
	An $\infty$-stack over $K$ is an object $\mathfrak{X}\in \Sh{Sch_{/K}}$. 
\end{definition}

By \cite[Lemma 2.11]{relativeTopHigherStacks}, there is a colimit preserving functor 
\[\Sh{\bullet_\et}\colon \Sh{Sch_{/K}}\to \mathfrak{Top}_\infty.\]
Hence, we can functorially assign to every $\infty$-stack $\mathfrak{X}$ an $\infty$-topos 
$\Sh{\mathfrak{X}_\et}$, which we think of as \etale sheaves over $\mathfrak{X}$.  
For an $\infty$-category $\mathcal{C}$, the category of pro-objects is defined in \cite[7.1.6.1]{LurieHTT}. It is the category of finite limits preserving functors $\mathcal{C}\to \mathcal{S}$ into the $\infty$-category of spaces, considered as a full subcategory of $Fun(\mathcal{C},\mathcal{S})$. The main property of the pro-category we use is the following 

\begin{proposition}\cite[Proposition 2.3]{relativeTopHigherStacks}
	\label{prop: left adjoint exist}
	Let $F \colon \mathcal{C}\to \mathcal{D}$ be a functor preserving finite limits. Then the induced functor $\pro(F)\colon \pro(\mathcal{C})\to \pro(\mathcal{D})$ admits a left adjoint. 
\end{proposition}
From now on we shall abuse notation by denoting $\pro(F)$ just by $F$.

Let $f\colon\mathcal{T}\to \mathcal{U}$ be a geometric morphism of $\infty$-topoi.
By definition, the morphism $f^*\colon \m{U}\to \m{T}$ preserves finite limits, hence $\pro(f^*)$ admits a left adjoint which we denote $f_\sharp$. 
Hence, we get functors 
\[f_*,f_\sharp \colon\pro(\m{T}) \adj \pro(\m{U}) \colon f^*\] 
fitting into a pair of adjunctions 
\[
f_\sharp \dashv f^* \dashv f_*.
\]

Recall the definition of the relative homotopy type of a geometric morphism $f$
\begin{definition}
	\label{relTopRal}
	Let $f \colon \m{T}\to \m{U}$ be a geometric morphism of $\infty$-topoi.
	The \emph{relative topological realization of $\m{T}$ over $\m{U}$} 
	is 
	\[  \Et(f)=f_\sharp (\ast_\mathcal{T}) \in \pro(\m{U}), \]
	where $\ast_{\m{T}}$ is the terminal object of $\mathcal{T}$, considered as a constant pro-object.
\end{definition}

\begin{remark}
	In the case where $\m{T}$ is discrete, the above definition coincides with the definition given in \cite{BARNEA2016784}, 
	as the $\infty$-category $\pro{\m{T}}$ is equivalent to the $\infty$- category
	which results from the Barnea-Schlank model structure, as is proven in \cite{proCats}. 
\end{remark}

In \cite{higherStacks}, Carchedi defines the \etale profinite homotopy type of an $\infty$-stack. 
In \cite{topTypeStacks}, Cough defines a relative version of the topological type for a morphism of a $\infty$-stack to a scheme.
For our application, we need a relative version of the \etale topological type for a morphism of $\infty$-stacks. 
Namely, given a morphism of $\infty$-stacks $f\colon\mathfrak{X} \to \mathfrak{Y}$, we wish to define its relative topological realization. 

By the functoriality of $\Sh{\bullet_\et}$, we have a geometric morphism, denoted abusively by  
\[f\colon \Sh{\mathfrak{X}}\to \Sh{\mathfrak{Y}}.\] 

Applying the general theory of relative topological realization, we obtain a pro-sheaf \[\Et(f)=f_\sharp (\ast_{\mathfrak{X}}) \in \Prosh{\mathfrak{Y}_\et}.\]

Our aim now is to give a formula for the relative topological realization of a colimit of a morphism of $\infty$-stacks in terms of the topological realizations of the components. 
To do this we shall recall first some general categorical constructions associated with diagrams of $\infty$-topoi and $\infty$-categories in general. 

%%%%%%%%%%%%%%%%%%%%%%%%%%%%%%%%%%%%%%%%%%%%%%%%%%%%%%%%%%%%%%
\subsection{Adjunctions and Limits of Infinity Categories}
\label{subsection: Abstract Nonsense}
%%%%%%%%%%%%%%%%%%%%%%%%%%%%%%%%%%%%%%%%%%%%%%%%%%%%%%%%%%%%%%
Here we recall and esxpand the results of \cite{AdjDescent}, which we use in our computation of relative topological type of morphisms of stacks. 

Let $I$ be a simplicial set.
Consider a functor $C_\bullet\colon I\to Cat_\infty$, i.e. a diagram of $\infty$-categories of shape $I$.   
In this case, one can form the inverse limit $\invlim C_\bullet$ and the lax limit $\laxlim C_\bullet$ (see e.g. \cite[Definition 4.1]{AdjDescent}). 
There is also a notion of op-lax limit for this diagram, which is just  
\[
\oplaxlim C_\bullet \colon= (\laxlim C_\bullet^{op})^{op}.
\]
Both limits comes equipped with fully faithful embeddings $\invlim C_\bullet \to \laxlim C_\bullet$ and $\invlim C_\bullet \to \oplaxlim C_\bullet$. 

\begin{proposition}\cite[Proposition 5.1]{AdjDescent}
	\label{proposition: Lax adjoint exist}
	Let $\phi_\bullet\colon C_\bullet \to D_\bullet$ be a morphism of $I$ diagrams in $Cat_\infty$. 
	If for every $i\in I$ the functor $\phi_i \colon C_i \to D_i$ admits a right (resp. left) adjoint, then the induced functor $\laxlim (\phi_\bullet)\colon \laxlim C_\bullet \to \laxlim D_\bullet$ (resp. $\oplaxlim (\phi_\bullet) \colon \oplaxlim C_\bullet \to \oplaxlim D_\bullet$) admits a right (resp. left) adjoint.  
\end{proposition}

\begin{remark}
	The result in $\cite{AdjDescent}$ is stated only for right adjoints. The case of left adjoints follows in a similar (but dual) manner.  
\end{remark}

%Now suppose we are given a cone $\phi_\bullet: C\to D_\bullet$, i.e. a morphism from the constant diagram of shape $I$ and value $C$ to the diagram $D_\bullet$. 

%\begin{proposition} \cite[Theorem B and Corollary 1.3]{AdjDescent}
%\label{proposition: Adjoint descent argument}
%If each $\phi_i : C\to D_i$ admits a right (resp. left) adjoint,  and $C$ admits all $I$-shaped limits (resp. $I^{op}$-shaped colimits) then the comparison map $ C\to\invlim D_\bullet$ admits a right (resp. left) adjoint.
%\end{proposition}
%From now on we shall refer to this result as the \emph{adjoint descent argument}. 

%In fact, the left adjoint of the comparison map is explicitly described as the composition 
%\[\invlim D_\bullet \to \oplaxlim D_\bullet \to \oplaxlim C \cong C^{I^{op}} \to C,\]
%where the left map is the natural fully faithful mbedding, the middle is the left adjoint to $\oplaxlim(\phi_\bullet)$ and the last is the colimit functor in $C$. 

As in \cite[Remark 5.2]{AdjDescent}, even if $f_\bullet\colon C_\bullet \to D_\bullet$ admits right (or left) adjoints level-wise, then the induced functor $\invlim(\phi_\bullet)\colon \invlim C_\bullet \to \invlim C_\bullet$ might not admit a right (or left) adjoint. However, under some extra assumptions it does. 

\begin{definition}
	Let 
	\[\xymatrix{
		A \ar^{u}[d]\ar^{f}[r] & B\ar^{v}[d]   \\
		C \ar^{g}[r] & D  \\
	}\]
	be a commutative square in $Cat_\infty$ (i.e. a natural isomorphism $gu\stackrel{\sim}{\to} vf$). 
	Suppose that $u$ and $v$ admit left adjoints $L_u$ and $L_v$. Let 
	$BC\colon L_vg \to fL_u$ denote the \emph{Beck-Chevalley} natural transformation. If $BC$ is an equivalence, we say that the commutative square satisfies the \emph{left Beck-Chevalley condition} (or, in short, left BC-condition).   
\end{definition} 

One can easily define the dual notion of right BC-condition, in case where $u,v$ admit right adjoints. As a matter of convention, the BC-condition shall always refer to the left (or right) adjoints of the \emph{vertical} maps. 

\begin{lemma}
	\label{lemma: two-sided restriction of adjunctions}
	Let $F\colon C\adj D\colon G$ be an adjunction between $\infty$-categories. Let $C'\subseteq C$ and $D'\subseteq D$ be full subcategories, such that $F|_{C'}$ factors through $D'$ and $G|_{D'}$ factors through $C'$. Then $G|_{D'}$ is a right adjoint to $F|_{C'}$ and the square  
	\[
	\xymatrix{
		C'\ar@{^{(}->}[r]\ar^{F|_{C'}}[d] & C\ar^{F}[d] \\ 
		D'\ar@{^{(}->}[r] & D
	}
	\]
	satisfies the right Beck-Chevalley condition. 
	Similarly, the canonical square 
	\[
	\xymatrix{
		D'\ar@{^{(}->}[r]\ar^{G|_{D'}}[d] & D\ar^{G}[d] \\
		C'\ar@{^{(}->}[r] & C \\ 
	}
	\]
	satisfies the left Beck-Chevalley condition. 
\end{lemma}

\begin{proof}
	First, clearly $G|_{D'}$ is right adjoint to $F|_{C'}$ (compare \cite[Lemma 5.4]{AdjDescent}). 
	We shall show that the first square satisfies the Beck-Chevalley conditions, as the second follows analogously. 
	Let $u\colon id_{C} \to GF$ and $c: FG \to id_{D}$ denote the unit and counit of the adjunction respectively. Let $i_C\colon C'\into C$ and $i_D \colon D'\into D$ denote the fuly faithful embeddings.
	The Beck-Chevalley map of the first square is, by definition, the composition 
	\[
	Gi_D \stackrel{(Gi_D) c|_{D'}}{\to} Gi_D (F|_{C'} G|_{D'}) \cong (G F) i_C G|_{D'} \stackrel{u(i_C G|_{D'})}{\to} i_C G_{D'}
	\]
	When applied to an object $d\in D' \subseteq D$, up to identifications of objects in a full subcategory with their image in the embient category, this is just the composition 
	$G(d)\stackrel{u}{\to} GFG(d)= GFG(d)\stackrel{c}{\to} G(d)$. This composition is an equivalence by the zygzag indentity for the adjunction $F \vdash G$.  
\end{proof}

\begin{remark}
	Note that, clearly, the Beck-Chevalley map of the first square is precisely the natural transformation rendering the second diagram commutative and vice versa.
\end{remark}

\begin{definition}
	Let $f_\bullet\colon C_\bullet \to D_\bullet$ be a natural transformation of $I$-shaped diagram of $\infty$-categories. We say that $f_\bullet$ satisfies the right (resp. left) BC-condition, if for every $i\in I$ the functor $f_i$ admits right (resp. left) adjoint and for every $e\colon i\to j$ in $I$ the resulting commutative square 
	\[
	\xymatrix{
		C_i \ar^{e_*}[r]\ar^{f_i}[d] &  C_j \ar^{f_j}[d]\\
		D_i \ar^{e_*}[r] & D_j 
	}
	\]
	satisfies the right (resp. left) Beck-Chevalley condition.
\end{definition}

Before we state our main usage of Beck-Chevalley transformations, let us discuss a little bit the construction of the right (or left) adjoint on the level of lax limits. As usual we shall discuss the case where the transformation has right adjoint, leaving the dual varification to the reader.  
If $f_\bullet :C_\bullet \to D_\bullet$ is a natural transformation of $I$-shaped diagrams of $\infty$-categories, and if $f_i \colon C_i \to D_i$ has a right adjoint $g_i$ for every $i\in I$, then $f_{lax}$ has a right adjoint $g_{lax}$. Let $\int_I C\bullet \to I$ and $\int_I D_\bullet \to I$ denote the respective coCartesian fibrations over $I$. Then $g_{lax}$ is obtained from a functor $G: \int_I D_\bullet \to \int_I C_\bullet$ (over $I$) by applying $G$ to sections of the structure map $\int_I D_\bullet \to I$. 
This functor restricts for every $i\in I$ to a functor $G_i : D_i \to C_i$ which is naturally identified with $g_i$, see \cite[Proposition 5.1]{AdjDescent}. By inspection of the construction of $G$ as in \cite[Section 5]{AdjDescent}, one can identify also how $G$ acts on morphisms. Every morphism in $\int_I D_\bullet$ canonically factors as a composition of a coCartesian morphism and a morphism lying in a single $D_i$. Hence, given the equivalence $G|_{D_i} \simeq g_i$, it suffices to describe the application of $G$ on coCartesian edges.  
Let $i,j\in I$ and $e\colon i\to j$ be an edge. Then we have induced functors $e_* : D_i \to D_j$ and $e_* : C_i \to C_j$. The fact that $f$ is a natural transformation gives us a commutative square 
\[
\xymatrix{
	C_i \ar^{e_*}[r]\ar^{f_i}[d] &  C_j \ar^{f_j}[d]\\
	D_i \ar^{e_*}[r] & D_j 
}
.\]
For every $d \in D_i \subseteq \int_I D_\bullet$ we have an essentially unique coCartesian edge $\psi: d \to e_* d$. Then, $G(\psi)$ is the composition of the unique coCartesian edge $\psi'\colon g_i(d)\cong G(d)\to e_* G(d)\cong e_* g_i(d)$ with the Beck-Chevalley map of the square above, namely the map 
\[BC : e_* g_i(d) \to g_j e_*(d)\cong G(e_*(d)).\]

\begin{proposition}
	\label{theorem: BC map then has left adjoint}
	Let $f_\bullet\colon C_\bullet \to D_\bullet$ be a morphism of diagrams of $\infty$-categories. If $f_\bullet$ satisfies the right (resp. left) Beck-Chevalley condition, then the induced functor 
	$\invlim f_\bullet\colon \invlim C_\bullet \to \invlim D_\bullet$ admit a right (resp. left) adjoint. 
	Moreover, the canonical commutative square
	\[
	\xymatrix{
		\invlim C_\bullet \ar@{^{(}->}[r] \ar^{\invlim f_\bullet}[d]& \laxlim C_\bullet \ar^{\laxlim f_\bullet}[d]\\  
		\invlim D_\bullet \ar@{^{(}->}[r]& \laxlim D_\bullet 
	}
	\]
	(resp. the square 
	\[
	\xymatrix{
		\invlim C_\bullet \ar@{^{(}->}[r] \ar^{\invlim f_\bullet}[d]& \oplaxlim C_\bullet \ar^{\oplaxlim f_\bullet}[d]\\  
		\invlim D_\bullet \ar@{^{(}->}[r]& \oplaxlim D_\bullet 
	}
	\]
	)
	satisfies the right (resp. left) Beck-Chevalley condition. 
\end{proposition}

\begin{proof}
	We do the right case, as the left case is completely analogous. 
	Let $g_i$ denote the right adjoint to $f_i$. 
	By Lemma \ref{lemma: two-sided restriction of adjunctions},
	it suffices to prove that the restriction to $\invlim D_\bullet$ of the right adjoint to $\laxlim f_\bullet$ factors through $\invlim C_\bullet$. 
	Namely, it would suffice to show that, in the notation of the discussion above this proposition, $G$ sends coCartesian edges to coCartesian edges. If $\psi\colon d\to d'$ is coCartesian, then $G(\psi)$ is a composition of a coCartesian edge with the Beck-Chevalley map. By the assumption, those Beck-Chevalley maps are isomorphisms, so $G$ send coCartesian edges to coCartesian edges.   
	%We do the case of right BC-condition, the other case is completely analogous. 
	%By \cite[Lemma 5.4]{AdjDescent}, it suffices to show the right adjoint to $\laxlim f_\bullet$ restricts to a functor $\invlim D_\bullet \to \invlim C_\bullet$. The essential image of $\invlim C_\bullet$ in $\laxlim C_\bullet$ consist of those sections of the associated co-Cartesian fibration $\int C_\bullet \to I$ that take edges of $I$ to co-Cartesian edges of $\int C_\bullet$, or in other words co-Cartesian sections. Hence it suffices to prove that the right adjoint to $\laxlim(f_\bullet)$ send co-Cartesian sections to co-Cartesian sections. 
	%This property can be checked for each edge $e:i\to j$ separately, and for each such edge correspond precisely for the right BC-condition of the square 
	%\[
	%\xymatrix{
	%C_i\ar^{e_*}[r]\ar^{f_i}[d] & C_j\ar^{f_j}[d] \\
	%D_i\ar^{e_*}[r] & D_j \\ 
	%}
	%.\]
	
	%Indeed, Let $g_i$ denote the left adjoint to $f_i$. It is immediate to verify that the left adjoint to $\laxlim(f_\bullet)$ takes a section $d_\bullet \in \laxlim D_\bullet$ to a a section $c_\bullet \in \laxlim C_\bullet$ such that $c_i = g_i (d_i)$, and an edge $e:i\to j$ to an edge $\tilde{e}\in Hom_{\int C_\bullet}(c_i,c_j)$ which correspond to the BC-map            
\end{proof}

We now turn to discuss descent properties of the Beck-Chevalley condition from diagrams to their limits. 

\begin{proposition}
	\label{proposition: lax lim beck chevalley descent}
	Let $I$ be a small $\infty$-category.
	Let
	\[\xymatrix{
		A_\bullet \ar^{f_\bullet}[r] \ar^{u_\bullet}[d]& \ar^{v_\bullet}[d] B_\bullet \\
		C_\bullet \ar^{g_\bullet}[r] & D_\bullet
	}
	\]
	be a commutative square of $I$-shaped diagrams of $\infty$-categories. 
	If for every $i\in I$ the square
	\[\xymatrix{
		A_i \ar^{f_i}[r] \ar^{u_i}[d]& \ar^{v_i}[d] B_i \\
		C_i \ar^{g_i}[r] & D_i
	}
	\]
	
	satisfies the right (resp. left) Beck-Chevalley condition, then
	the square
	
	\[\xymatrix{
		\laxlim A_\bullet \ar^{\laxlim f_\bullet}[r] \ar^{\laxlim u_\bullet}[d]& \ar^{\laxlim v_\bullet}[d] \laxlim B_\bullet \\
		\laxlim C_\bullet \ar^{\laxlim g_\bullet}[r] & \laxlim D_\bullet
	}
	\] 
	(resp. the square
	\[\xymatrix{
		\oplaxlim A_\bullet \ar^{\oplaxlim f_\bullet}[r] \ar^{\oplaxlim u_\bullet}[d]& \ar^{\oplaxlim v_\bullet}[d] \oplaxlim B_\bullet \\
		\oplaxlim C_\bullet \ar^{\oplaxlim g_\bullet}[r] & \oplaxlim D_\bullet
	}
	\] 
	)
	satisfies the right (resp. left) Beck-Chevalley condition. 
\end{proposition}

\begin{proof}
	It is easy to see that the Beck-Chevalley map of the square of lax limits restricts at every $i \in I$ to the Beck-Chevalley map of the square 
	\[\xymatrix{
		A_i \ar^{f_i}[r] \ar^{u_i}[d]& \ar^{v_i}[d] B_i \\
		C_i \ar^{g_i}[r] & D_i
	}
	\] 
	The result now follows from the fact that equivalences in $\laxlim B_\bullet$ and $\oplaxlim B_\bullet$ are those maps which restrict to an equivalence at every $i\in I$.
\end{proof}
%Recall that [Refference], if we have a diagram 
%\[
%\xymatrix{
%A \ar[r]\ar[d] & B \ar[r]\ar[d] & C \ar[d] \\ 
%D \ar[r] &  E \ar[r] & F 
%}
%\]

Before we state the descent result of BC-conditions, let us recall some pasting conditions for BC-squares. 
Suppose we are given a commutative diagram 
\[
\xymatrix{
	A \ar^{a}[r] \ar^{f}[d] & B \ar^{b}[r] \ar^{g}[d] & C\ar^{h}[d] \\ 
	A' \ar^{a'}[r] & B' \ar^{b'}[r] & C'  
}
\]
of $\infty$-categories for which all the vertical functors admit right adjoints. Then, the BC-map of the outer square is the composition of the BC-maps of the two smaller squares, namely the right BC map of the outer square is given by 
\[
ba R_f \stackrel{b BC}{\to} bR_ga \stackrel{BC a'}{\to}R_hb'a'
.\]

In particular, if the two small squares satisfies the right BC-condition so is the outer square. Moreover, if $b$ is conservative, the outer and the left squares satisfies the right BC condition, then so does the right square. 
We call those two properties \emph{horizontal pasting} of Beck-Chevalley squares. (The case for the left BC condition is similar).

\begin{theorem}
	\label{theorem: inverse limits Beck Chevalley descent}
	Let $I$ be a small $\infty$-category and let
	\[\xymatrix{
		A_\bullet \ar^{f_\bullet}[r] \ar^{u_\bullet}[d]& \ar^{v_\bullet}[d] B_\bullet \\
		C_\bullet \ar^{g_\bullet}[r] & D_\bullet
	}
	\]
	be a commutative square of $I$-shaped diagrams of $\infty$-categories. Suppose that $u_\bullet$ and $v_\bullet$ satisfy the right (resp. left) Beck-Chevalley condition, and for every $i\in I$ the restricted square
	\[\xymatrix{
		A_i \ar^{f_i}[r] \ar^{u}[d]& \ar^{v}[d] B_i \\
		C_i \ar^{g}[r] & D_i
	}
	\]
	satisfies the right (resp. left) Beck-Chevalley condition. Then the square 
	\[\xymatrix{
		\invlim A_\bullet \ar^{\invlim f_\bullet}[r] \ar_{\invlim u_\bullet}[d]& \ar^{\invlim v_\bullet}[d] \invlim B_\bullet \\
		\invlim C_\bullet \ar^{\invlim g_\bullet}[r] & \invlim D_\bullet
	}
	\]
	satisfies the right (resp. left) Beck-Chevalley condition. 
\end{theorem}

\begin{proof}
	We do the right case, the left follows analogously. 
	Consider the commutative cube 
	\[\xymatrix{
		\invlim A_\bullet \ar^(0.5){\invlim f_\bullet}[rr] \ar_{\invlim u_\bullet}[dd] \ar@{^{(}->}[rd]&&   \ar^(0.7){\invlim v_\bullet}[dd] \invlim B_\bullet \ar@{^{(}->}[rd] & \\
		& \laxlim A_\bullet \ar^(0.3){\laxlim f_\bullet}[rr] \ar^(0.3){\laxlim u_\bullet}[dd]&         & \laxlim B_\bullet \ar^{\laxlim v_\bullet}[dd]  \\
		\invlim C_\bullet \ar^(0.3){\invlim g_\bullet}[rr] \ar@{^{(}->}[rd] && \invlim D_\bullet \ar@{^{(}->}[rd]& \\ 
		&\laxlim C_\bullet \ar^(0.5){\laxlim g_\bullet}[rr]&                   & \laxlim D_\bullet 
		\\
	}
	\]
	The right and left faces satisfies the right Beck-Chevalley condition by Proposition \ref{prop: left adjoint exist}. The front face satisfies the right Beck-Chevalley condition by Proposition \ref{proposition: lax lim beck chevalley descent}. 
	Since the inclusion of the lax limit to the limit is fully faithful, to show that the back-face satisfies the right BC-condition, its enough to show that its horizontal pasting with the right face satisfies the right BC-condition. By the commutativity of the diagram, this is the same as the BC-condition for the horizontal pasting of the left face and the front face. Since these two satisfy the right Beck-Chevalley condition, we are done. 
\end{proof}

\subsection{Relative Homotopy Type of Colimits of Infinity Topoi}
Let $I$ be a small $\infty$-category and let $f_\bullet\colon \m{T}_\bullet \to  \m{U}_\bullet$ be a morphism of $I$-shaped diagrams in $\mathfrak{Top}_\infty$.
By definition of colimits in $\mathfrak{Top}_\infty$, the colimits $\colim \m{T}_\bullet$ and $\colim \m{U}_\bullet$ are actually the limits, computed with respect to the left adjoints corresponding to the structure geometric morphisms \cite[Proposition 6.3.3.1]{LurieHTT}.
Hence, we can apply the results of Subsection \ref{subsection: Abstract Nonsense} to give a formula for the relative homotopy type of the colimit $\colim f_\bullet$. 

The diagrams $\m{T}_\bullet$ and $\m{U}_\bullet$ can be considered as $I^{op}$-shaped diagrams in $Cat_\infty$ via the forgetful functor $\mathfrak{Top}_\infty^{op} \to Cat_\infty$ \cite[Definition 6.3.1.5]{LurieHTT}. Denote by $(\m{T}_\bullet)^*$ and $(\m{U}_\bullet)^*$ the resulting $I^{op}$-shaped diagrams in $Cat_\infty$. The geometric natural transformation $f_\bullet \colon \m{T}_\bullet \to \m{U}_\bullet$ induces a natural transformation $f_\bullet^* \colon (\m{U}_\bullet)^* \to (\m{T}_\bullet)^*$. 

Let $f=\colim f_\bullet$, $\m{T}=\colim_I \m{T}_\bullet=\invlim_{I^{op}} (\m{T}_\bullet)^*,$ and $\m{U}=\colim_I \m{U}_\bullet=\invlim_{I^{op}} (\m{U}_\bullet)^*$.  
In this case, the functors $f_i^* \colon \pro(\m{U}_i) \to \pro(\m{T}_i)$ admits left adjoints, and hence by Proposition \ref{proposition: Lax adjoint exist} the op-lax functor $f_{lax}^*$ associated with $f_\bullet^*$  admits a left adjoint 
\[f_{lax,\sharp} \colon\oplaxlim \pro(\m{T}_\bullet)^* \to \oplaxlim \pro(\m{U}_\bullet)^*.\]   

Let $\rho_\bullet\colon \m{T}_\bullet\to \m{T}$ and $\nu_\bullet\colon\m{U}_\bullet \to  \m{U}$ denote the comparison geometric morphisms, where the targets are considered as constant $I$-shaped diagrams. 
We can now form the following natural commutative diagram 

\begin{equation}
\xymatrix{
	\oplaxlim \pro(\m{T}_\bullet)^* & \pro(\m{T})^{I}\ar^(.4){\rho_{lax}^*}[l] & \pro(\m{T}) \ar^{\pi^*}[l] \\ 
	\oplaxlim \pro(\m{U}_\bullet)^* \ar^{f_{lax}^*}[u] & \pro(\m{U})^{I} \ar^{(f^*)^I}[u]\ar^(.4){\nu_{lax}^*}[l] &  
	\pro(\m{U}) \ar^{f^*}[u]\ar^{\pi^*}[l] \\ 
	%\oplaxlim Pro(\m{T}_\bullet)^*          && \oplaxlim Pro(\m{U}_\bullet)^* \ar^{f_\bullet^*}[ll] \\ 
	%Pro(\m{T})^{I}\ar^{\rho_\bullet^*}[u]   && Pro(\m{U})^{I} \ar^{(f^*)^I}[ll]\ar^{{\nu_\bullet}^*}[u] \\ 
	%Pro(\m{T})  \ar^{\pi^*}[u]              && Pro(\m{U}) \ar^{f^*}[ll]\ar^{\pi^*}[u] 
}
\end{equation} 
where $\pi\colon I^{op}\to pt$ denote the projection, and $\pi^*$ is pre-composition with $\pi$.
Let $\tilde{\rho}^* = \rho_{lax}^* \pi^*$ and $\tilde{\nu}^* = \nu_{lax}^* \pi^*$. Let $\tilde{\rho}_\sharp$ and $\tilde{\nu}_\sharp$ denote the respective left adjoints.  
Hence, we have a commutative diagram
\begin{equation}
\label{equation: right adjoints diagram adjoint descent}
\xymatrix{
	\oplaxlim \pro(\m{T}_\bullet)^*  & \pro(\m{T}) \ar^(.4){\tilde{\rho}^*}[l] \\ 
	\oplaxlim \pro(\m{U}_\bullet)^* \ar^{f_{lax}^*}[u] &   \pro(\m{U}) \ar^{f^*}[u]\ar^(.4){\tilde{\nu}^*}[l]  
}
\end{equation}
and as a result we obtain a Beck-Chevalley map 
$BC_\sharp: f_{\bullet,\sharp} \tilde{\rho}^*\to \tilde{\nu}^* f_\sharp$. Taking the mate of this map we obtain a natural transformation 
\[\mu\colon \tilde{\nu}_\sharp f_{lax,\sharp} \tilde{\rho}^*\to  f_\sharp.\]

\begin{theorem}
	\label{theorem: formula for push from colimit}
	The natural transformation $\mu$ is an equivalence at every object of the full subcategory $\m{T}\subseteq \pro(\m{T})$. 
\end{theorem} 

\begin{proof}
	Let $x\in \m{T}$. We have to show that 
	\[\mu_x\colon
	\tilde{\nu}_\sharp f_{lax,\sharp} \tilde{\rho}^*(x)\to  f_\sharp (x)\]
	is an equivalence. Since the objects of $\m{U}$ co-generate the pro-category $\pro(\m{U})$, it suffices to prove that for every $y\in \m{U}$, the induced map 
	\[\mu_x^*\colon
	\Hom_{\pro(\m{U})}(f_\sharp(x),y)\to 
	\Hom_{\pro(\m{U})}(\tilde{\nu}_\sharp f_{lax,\sharp} \tilde{\rho}^*(x),y). 
	\]
	is an equivalence.
	Consider the following diagram 
	\[
	\tiny{
		\xymatrix{ 
			\Hom_{\pro(\m{U})}(f_\sharp x,y) \ar^{\mu_x^*}[rr]&& \Hom_{\pro(\m{U})}(\tilde{\nu}_\sharp f_{lax,\sharp} \tilde{\rho}^*x,y)\\
			&&  \Hom_{\oplaxlim \pro(\m{U}_\bullet)^*}(f_{lax,\sharp}\tilde{\rho}^*x,\tilde{\nu}^* y)\ar^{\sim}[u]\\
			\Hom_{\pro(\m{T})}(x,f^*y)\ar^{\sim}[uu] \ar^(.4){(\tilde{\rho}^*)}[r]\ar_{\tilde{\rho}^*}[rd]&\Hom_{\oplaxlim \pro(\m{T}_\bullet)^*}(\tilde{\rho}^*x,\tilde{\rho}^*f^*y)\ar^{\sim}[r]&  \Hom_{\oplaxlim \pro(\m{T}_\bullet)^*}(\tilde{\rho}^*x,f_{lax}^*\tilde{\nu}^* y)\ar^{\sim}[u]\\
			&\Hom_{\invlim \pro(\m{T}_\bullet)^*}(\tilde{\rho}^* x,\tilde{\rho}^*f^*y)\ar^{\sim}[r]\ar^{\sim}_{\beta}[u]&  \Hom_{\invlim \pro(\m{T}_\bullet)^*}(\tilde{\rho}^*(x),f_{lax}^*{\tilde{\nu}}^* y)\ar^{\sim}_{\beta}[u] \\
			\Hom_{\m{T}}(x,f^*y)\ar^(.4){\tilde{\rho}^*}[r]\ar^{\sim}_{\alpha}[uu] &\Hom_{\invlim\m{T}_\bullet^*}(\tilde{\rho}^*x,\tilde{\rho}^*f^*y)\ar^{\sim}[r]\ar^{\sim}_{\alpha}[u]&  \Hom_{\invlim \m{T}_\bullet^*}((\tilde{\rho}^*x,f_{lax}^*\tilde{\nu}^*y) \ar^{\sim}_{\alpha}[u] 
	}}
	\]
	in which the arrows labeled $\alpha$ and $\beta$ come from the natural fully faithful embeddings $Id_{Cat_\infty} \to \pro$ and $\invlim \to \oplaxlim$, the horizontal unlabeled arrows arises from the commutativity of the diagram (\ref{equation: right adjoints diagram adjoint descent}), and the vertical equivalences all come from the various adjunctions. 
	
	The upper big rectangle commute by the definition of $\mu$. The lower left trapezoid and the two lower right rectangles commute by the naturality of $\alpha$ and $\beta$. Finally, the triangle at the middle-left commute by the definition of $\tilde{\rho}^*$ (it factors uniquely through the actual limit). 
	
	Hence, the big outer square commute up to homotopy, and $\mu_x^*$ is an equivalence if and only if the bottom left horizontal map   
	\[\Hom_{\m{T}}(x,f^*y) \stackrel{\tilde{\rho}^*}{\to}\Hom_{\invlim\m{T}_\bullet^*}(\tilde{\rho}^*x,\tilde{\rho}^*f^*y)\] is an equivalence. But $\tilde{\rho}^*$ is induced from the comparison map $\m{T}\to \invlim \m{T}_\bullet^*$ which is an equivalence by the assumption. 
\end{proof}

From this result, one immediately get as a corollary a formula for the relative homotopy type of a colimit of morphisms of $\infty$-topoi. 
Let $f_\bullet: \m{T}_\bullet \to \m{U}_\bullet$ be a morphism of $\infty$-topoi, with colimit $f:\m{T}\to \m{U}$. Then, with the notation as above, we have a well defined element $f_{lax,\sharp} \tilde{\rho}^* \ast_{\m{T}} \in \oplaxlim \m{U}_\bullet^*$. Denote this element by $|f_\bullet|$. Then, by Theorem \ref{theorem: formula for push from colimit} we obtain: 

\begin{corollary}
	\label{corollary: formula for relative type colimit} 
	\[\tilde{\nu}_\sharp |f_\bullet|\cong |f|\]
\end{corollary}

Informally, this means that the relative homotopy type of the colimit is the colimit of the relative homotopy types. 
\subsection{The Relative \Etale Homological Type}
% % % % % % % % % % % % % % %
%%%%%%%%%%%%%%%%%%%%%%%%%%%%%

\label{relToSheavesOfComplexes}
In this section we shall recall the theory of sheaves of modules over a ring in a general $\infty$-topos. Note that the case of classical topos and in particular of classical stacks is well studied, and in particular all the results of this section are well known for the case of a classical topoi and stacks. The goal is therefore to promote parts of the classical theory to the $\infty$-categorical and higher stacks settings. 

For a ring $\Lambda$, let $D(\Lambda)$
denote the stable $\infty$-category of complexes of $\Lambda$-modules. Let $\mathcal{S}$ be the $\infty$-category of spaces.
There is an adjunction 
\begin{equation}
\label{equation: free forgetful adjunction}
I_\Lambda \colon \mathcal{S} \adj D(\Lambda) \colon R_\Lambda
\end{equation}
where $R_\Lambda$ denotes the functor $\Hom_{D(\Lambda)}(\Lambda,\bullet)$ and we have
$\pi_n(I_\Lambda(X)) = \HH_n(X,\Lambda)$. 

Let $\prl$ be the $\infty$-category of presentable $\infty$-categories with colimit preserving functors between them .
Then $\prl$ is a symmetric monoidal $\infty$-category as in \cite[Section 4.8]{LurieHA}. In fact, the symmetric monoidal structure satisfies $C\otimes D \cong \RFun(C^{op},D)$, the limits preserving functors from $C^{op}$ to $D$, by \cite[Proposition 4.8.1.17]{LurieHA}. The adjunction (\ref{equation: free forgetful adjunction}) can be seen as a morphism $I_\Lambda\colon \m{S}\to D(\Lambda)$ in $\prl$. 

By definition, every $\infty$-topos is presentable, and moreover we have a canonical functor 
\[\Psi\colon \mathfrak{Top}_\infty^{op}\to \prl\] 
which sends topoi with geometric morphisms between them to the underlying presentable categories with the left adjoint functors between them. 
Tensoring $\Psi$ with the morphism $\m{S}\stackrel{I_\Lambda}{\to} D(\Lambda)$ we obtain a natural transformation in $\prl$,
also denoted $I_\Lambda$, of the form 
\[I_\Lambda \colon \Psi \cong \m{S}\otimes \Psi \to D(\Lambda) \otimes \Psi.\] Denote $\Sh{\m{T},\Lambda}=D(\Lambda) \otimes \m{T}$. This is the $\infty$-category of sheaves of $\Lambda$-modules over $\m{T}$. If $\m{T} = \Sh{C}$ for a site $C$, then $\Sh{\m{T},\Lambda}$ is canonically identified with the $\infty$-category of sheaves of $\Lambda$-modules over $C$.   

%Evaluating the natural transformation $\Lambda \otimes \bullet$ on a morphism $f:\m{T}\to \m{T}'$ of $\infty$-topoi, we obtain a commutative diagram of $\infty$-categories 

%\begin{equation}
%\label{equation: pullback and tensoring commute}
%\xymatrix{
%\m{T}\ar^{\Lambda \otimes\bullet }[rr] & & Mod_\Lambda(\m{T}) \\
%\m{T}'\ar^{f^*}[u]\ar^{\Lambda \otimes\bullet }[rr] && Mod_\Lambda(\m{T}') \ar^{f^*}[u]\\
%}
%\end{equation}

%which is functorial in $f$ in the sence that it is ;the value at $f$ of a functor $Fun(\Delta^1,\mathfrak{Top}_\infty)\to Fun(\Delta^1\times \Delta^1,\prl)$. 
%Passing to pro-categories in this commutative square we obtain a canonical commutative square 

%\begin{equation*}
%\label{equation: pullback and tensoring commute}
%\xymatrix{
%Pro(\m{T})\ar^{\Lambda \otimes\bullet }[rr] & & Pro(Mod_\Lambda(\m{T})) \\
%Pro(\m{T}')\ar^{f^*}[u]\ar^{\Lambda \otimes\bullet }[rr] && Pro(Mod_\Lambda(\m{T}')) \ar^{f^*}[u]\\
%}
%\end{equation*}

%Since $f^* :\m{T}\to \m{T}'$ commutes with finite limits by definition, and $f^*:Mod_\Lambda(\m{T})\to Mod_\Lambda(\m{T}')$ commutes with finite limits since it is a left functor of stable $\infty$-categories, after prolongation both functors admit left adjoints denoted both by $f_\sharp$ by abuse of notation. Hence, we obtain a functor 

For every $\infty$-topos $\m{T}$ the $\infty$-category $\Sh{\m{T},\Lambda}$ is stable. 
Hence for a geometric morphism $f\colon\m{T}\to \m{U}$ 
the induced functor $f^*\colon\Sh{\m{T},\Lambda}\to \Sh{\m{U},\Lambda}$ preserve finite limits.
It follows that after prolongation it admits a left adjoint 
\[f_\sharp\colon \pro(\Sh{\m{U},\Lambda})\to \pro(\Sh{\m{T},\Lambda}).\]

We shall now define the notion of relative homological type of a geometric morphism of $\infty$-topoi. For an $\infty$-topos $\m{T}$ and a ring $\Lambda$ let \[\Lambda_\m{T}=I_\Lambda (*_\m{T}).\] 

\begin{definition}
	Let $f\colon\m{T}\to \m{U}$ be a geometric morphism of $\infty$-topoi. Define the \emph{relative homological type} of $f$ to be 
	\[|f|_\Lambda=f_\sharp \Lambda_\m{T}\in Pro(\Sh{\m{U},\Lambda}).\]
\end{definition}

We would like to compare now the relative homological and the homotopy types of a geometric morphism. Before we do that, we shall discuss some basic properties of the symmetric monoidal structure on $\prl$.

\subsubsection{Tensor Products of Presentable Infinity Categories and Beck-Chevalley Conditions}
Here we shall recall and discuss some basic properties of the tensor product of presentable $\infty$-categories.
Let $C,D\in \prl$ be two presentable $\infty$-categories.
In \cite[Section 4.8]{LurieHA} the equivalence of $\RFun(C^{op},D)$ and  $C\otimes D$ is proved by comparing the co-represented functors. Directly from this identification, it is easy to deduce part of the functoriality of of the right hand side as a functor $\prl \times \prl \to \prl$.

For $u\colon C\to C'$ a morphism in $\prl$ and $D\in \prl$, then the right adjoint to the composition 
\[\RFun(C^{op},D)\cong C\otimes D \stackrel{u\otimes id_D}{\to} C'\otimes D \cong \RFun((C')^{op},D)\] 
is easily identified with the morphism of pre-composition with $u^{op}$. 
Similarly, the right adjoint to the composition 
\[\RFun(D^{op},C)\cong D\otimes C \stackrel{id_D\otimes u}{\to} D\otimes C' \cong \RFun(D^{op},C')\]
is homotopic to the morphism of post-composition with the right adjoint to $u$.  

If $C$ is a presentable $\infty$-category, then we denote by $C^{\omega}$ the full subcategory of compact objects in $C$. We have a fully faithful embedding $\ind(C^{\omega})\into C$. The $\infty$-category $C$ is called compactly generated if this embedding is essentially surjective, and hence $C\cong \ind(C^{w})$. If $C$ and $D$ are $\infty$-categories we denote by $\Fun^{lex}(C,D)$ the full subcategory of $\Fun(C,D)$ spanned by the functors that commutes with finite limits.  

\begin{proposition}
	\label{proposition: compactly generated then restriction equivalence}
	Let $C$ be a compactly generated presentable $\infty$-category and let $D$ be a presentable $\infty$-category. Then restriction to the compact objects induces an equivalence 
	\[
	C\otimes D \cong \Fun^{lex}((C^{w})^{op},D)
	\]
\end{proposition}

\begin{proof}
	Recall that for presentable $\infty$-categories $A$ and $B$, restriction along the inclusion $A\to \pro A$ induces an equivalence $\RFun(\pro(A),B)\cong \Fun^{lex}(A,B)$. Hence, we have 
	\begin{align*}
	C\otimes D\cong \RFun(C^{op},D)\cong \RFun((\ind(C^{\omega}))^{op},D)\cong \\ 
	\cong \RFun(\pro ((C^{\omega})^{op}),D)\cong \Fun^{lex}((C^{\omega})^{op},D) 
	\end{align*}
	and the composition of those equivalences easily seen to be the one induced from the restriction along the inclusion $C^{\omega}\into C$. 
\end{proof}

%Let $C$ and $C'$ be presentable $\infty$-categories, and $u:C\to C'$ be a morphism in $\prl$, with right adjoint $v$.
%If $D$ is a presenable $\infty$-category, then we have an induced morphism $u\otimes id_D : C\otimes D \to C'\otimes D$ in $\prl$, namely a morphism $u\otimes id_D : RFun(C^{op},D)\to RFun((C')^{op},D)$.  

We shall now discuss Beck-Chevalley conditions for the exterior tensor product of morphisms in $\prl$. For this we need the following lemma.

\begin{lemma}
	\label{lemma: Beck Chevalle for post compostition}
	Let $F\colon C\adj C' \colon G$ be an adjunction and let $H\colon D \to D'$ be a functor. Then post-composition with the adjunction datum of $F$ and $G$ give rise to an adjunction 
	\[
	F \circ \colon \Fun(D,C) \adj \Fun(D,C') \colon  G \circ  
	.\] 
	Moreover, the square 
	\[
	\xymatrix{
		\Fun(D',C) \ar^{ F \circ}[d] \ar^{\circ H}[r] & \Fun(D,C) \ar^{ F \circ }[d] \\ 
		\Fun(D',C') \ar^{\circ H}[r] & \Fun(D,C')
	}
	\] 
	satisfies the right Beck-Chevalley condition, and the square  
	\[
	\xymatrix{
		\Fun(D',C') \ar^{\circ H}[r] \ar^{G \circ}[d] & \Fun(D,C') \ar^{G \circ}[d]\\
		\Fun(D',C)  \ar^{\circ H}[r] & \Fun(D,C) \\
	}
	\] 
	satisfies the left Beck-Chevalley condition.
\end{lemma}

\begin{proof}
	It is clear that $F\circ$ and $G\circ$ are adjoint by post-composing with the adjunction datum. We shall prove the Beck-Chevalley condition for the first square, as the second is completely analogous.
	Let $c\colon FG\to id_{C'}$ and $u\colon id_{C}\to GF$ denote the unit and the counit of the adjunction $F\dashv G$. 
	The Beck-Chevalley map of the first square has value on $\phi \colon D' \to D'$ is the composition 
	\[G(\phi H)  \stackrel{G u (\phi H) }{\to} (GF)G(\phi H)  \cong (G((FG)\phi)) H   \stackrel{G(c \phi) H}{\to} (G \phi) H.\]
	By the zygzag identities and the associativity of composition this map is homotopic to the associativity isomorphism $G (\phi H)\cong (G \phi) H $ and hence it is an equivalence. 
\end{proof}

Here again, the Beck-Chevalle map of one square is the homotopy rendering the second commutative and vice versa. 

\begin{proposition}
	\label{proposition: Beck-Chevalley tensor products} 
	Let $u\colon C \to C'$ and $v:D\to D'$ be morphisms in $\prl$. Assume that $C$ and $C'$ are compactly generated and that $u$ maps compact objects to compact objects. Suppose further that $v$ is left exact. Then the square 
	\[
	\xymatrix{
		C\otimes D \ar^{u\otimes id_D}[r]\ar_{id_C \otimes v}[d] & C' \otimes D \ar^{id_{C'} \otimes v}[d] \\ 
		C \otimes D' \ar^{u\otimes id_{D'}}[r]  & C'\otimes D'}
	\]
	satisfies the right Beck-Chevalley condition. 
\end{proposition}

\begin{proof}
	Let $R_v$ be the right adjoint to $v$. 
	It would suffice to prove that the square of right adjoints satisfies the left Beck-Chevalley condition. The square of right adjoints is canonically equivalent to the square 
	\[
	\xymatrix{
		\RFun((C')^{op},D') \ar^{u^{op}\circ}[r]\ar^{\circ R_v}[d]& \RFun((C)^{op},D')\ar^{\circ R_v}[d] \\ 
		\RFun((C')^{op},D) \ar^{u^{op}\circ}[r] & \RFun(C^{op},D) \\ 
	}
	.\]
	
	By Proposition \ref{proposition: compactly generated then restriction equivalence}, this square is equivalent to the square 
	\begin{equation}
	\label{euqation: square left exact functors diagram}
	\xymatrix{
		\Fun^{lex}(({C'}^{\omega})^{op},D') \ar^{u^{op}\circ}[r]\ar^{\circ R_v}[d]& \Fun^{lex}((C^{\omega})^{op},D')\ar^{\circ R_v}[d] \\ 
		\Fun^{lex}(({C'}^{\omega})^{op},D) \ar^{u^{op}\circ}[r] & \Fun^{lex}((C^{\omega})^{op},D) \\ }
	.\end{equation}
	
	The adjunction $v\colon D\adj D' \colon R_v$ gives rise to an adjunction between functors $Cat_\infty\to Cat_\infty$ of the form
	\[v \circ \colon \Fun(\bullet,D) \adj \Fun(\bullet, D')\colon   R_v \circ.\]
	Since $v$ is left exact the restriction of $v\circ $ to $\Fun^{lex}(\bullet, D)$ has its image in $\Fun^{lex}(\bullet, D')$, and hence by Lemma \ref{lemma: two-sided restriction of adjunctions} we get a restricted adjunction 
	\[v \circ \colon \Fun^{lex}(\bullet,D) \adj \Fun^{lex}(\bullet, D')\colon \circ R_v,\] 
	and the square 
	\[
	\xymatrix{
		\Fun^{lex}(\bullet, D) \ar^{\circ R_v}[d]\ar@{^{(}->}[r]& \Fun(\bullet,D)\ar^{\circ R_v}[d] \\ 
		\Fun^{lex}(\bullet,D')\ar@{^{(}->}[r] & \Fun(\bullet, D')
	}
	\]
	satisfies the left Beck-Chevalley condition. Hence, the left Beck-Chevalley condition of the square (\ref{euqation: square left exact functors diagram}) follows from the left Beck-Chevalley condition for the square 
	\[
	\xymatrix{
		\Fun(({C'}^{\omega})^{op},D') \ar^{u^{op}\circ}[r]\ar^{\circ R_v}[d]& \Fun((C^{\omega})^{op},D')\ar^{\circ R_v}[d] \\ 
		\Fun(({C'}^{\omega})^{op},D) \ar^{u^{op}\circ}[r] & \Fun((C^{\omega})^{op},D) \\ }
	\]
	(compare the proof of Theorem \ref{theorem: inverse limits Beck Chevalley descent}). This square satisfies the left Beck-Chevalley condition by Lemma \ref{lemma: Beck Chevalle for post compostition}. 
\end{proof}
As a corollary, we get
\begin{corollary}
	\label{corollary: compactly generated then BC for topoi}
	Let $f\colon\m{T}\to \m{U}$ be a geometric morphism of $\infty$-topoi and let $u\colon C\adj D \colon v$ be a morphism in $\prl$. Assume that $C$ and $D$ are compactly generated and that $u$ preserves compact objects. Then then the commutative square  
	\[\xymatrix{
		\m{U} \otimes C \ar^{id_\m{U}\otimes u}[r] & \m{U} \otimes D  \\ 
		\m{T} \otimes C \ar^{id_\m{T}\otimes u}[r]\ar^{f^*\otimes id_C}[u] & \m{T} \otimes D\ar_{f^*\otimes id_D}[u]
	}
	\]
	satisfies the right Beck-Chevalley condition.
\end{corollary}

\subsubsection{Comparison of the Relative Homological type and Homotopy Type}

We shall now discuss the compatibility of the relative homological and homotopical types for geometric morphisms. 
\begin{proposition}
	For every morphism of $\infty$-topoi $f\colon\m{T}\to \m{U}$, the diagram
	
	\[
	\xymatrix{
		\pro(\m{U}) \ar^{f_\sharp}[d] \ar^{I_\Lambda(\m{U})}[rr] \ar[d] && \pro(\Sh{\m{U},\Lambda}) \ar^{f_\sharp}[d]  \\ 
		\pro(\m{T}) \ar^{I_\Lambda(\m{T})}[rr]                    && \pro(\Sh{\m{T},\Lambda})
	}
	\]
	commutes up to homotopy. 
\end{proposition}

\begin{proof}
	
	Passing to the right adjoints, it would suffice to prove that the diagram
	\[
	\xymatrix{
		\pro(\m{U})   && \pro(\Sh{\m{U},\Lambda}) \ar^{R_\Lambda(\m{U})}[ll] \\ 
		\pro(\m{T}) \ar^{f^*}[u]        && \pro(\Sh{\m{T},\Lambda}) \ar^{f^*}[u] \ar^{R_\Lambda(\m{T})}[ll]
	}
	\]
	commutes up to homotopy. This diagram is the prolongation of 
	
	\[
	\xymatrix{
		\m{U}   && \Sh{\m{U},\Lambda} \ar^{R_\Lambda(\m{U})}[ll] \\ 
		\m{T} \ar^{f^*}[u]        && \Sh{\m{T},\Lambda} \ar^{f^*}[u] \ar^{R_\Lambda(\m{T})}[ll]
	}
	\]
	hence it suffices to prove that this square commutes up to homotopy.  
	
	The square
	\[
	\xymatrix{
		\m{U} \ar^{I_\Lambda(\m{U})}[d]  &  \m{T} \ar^{f^*}[l] \ar^{I_\Lambda(\m{T})}[d]  \\ 
		\Sh{\m{U},\Lambda}  & \Sh{\m{T},\Lambda} \ar^{f^*}[l] 
	}
	\]
	commutes by the naturality of $I_\Lambda$ and it would suffice to show that it satisfies the right Beck-Chevalley condition. This in turn follows from Corollary \ref{corollary: compactly generated then BC for topoi} since $I_\Lambda\colon D(\Lambda)\to \m{S}$ map compact objects to compact objects and $\mathcal{S}$ and $D(\Lambda)$ are compactly generated. 
\end{proof}

If we denote $\Lambda \otimes x=I_\Lambda(x)$ then the following 

\begin{corollary}
	\label{corollary: shrieck commutes with tensor}
	Let $f:\mathcal{T}\to \mathcal{U}$ be a geometric morphism, and let $\Lambda$ bea ring. Then 
	\[\Lambda \otimes f_\sharp(x)\cong f_\sharp(\Lambda \otimes x).\]   
	In particular, 
	\[|f|_\Lambda \cong |f|\otimes \Lambda.\]
\end{corollary}

\subsubsection{Cohomological Remark}

We conclude with a remark regarding various cohomological invariants. If $\m{T}$ is an $\infty$-topos and $x\in \m{T}$ we can define the cohomology of $x$ in a priori two different ways. On one hand, we can set 
\[H^n(x,\Lambda)=\pi_0 \Hom_{\Sh{\m{T},\Lambda}}(x\otimes \Lambda,\Lambda_\m{T}[n]),\] 
which is analogous to the definition using singular cochains of a space.
On the other hand, we can define the cohomology using maps to Eilenberg-Mclane spaces,
by 
\[H^n(x,\Lambda)=\Hom_\m{T}(x,\Gamma_\m{T}^*K(\Lambda,n))\] 
where $\Gamma_\m{T}\colon\m{S}\to \m{T}$ is the unique geometric morphism.
However, since \[\Gamma_{\m{T}}^*K(\Lambda,n)=\Gamma_{\m{T}}^*R_\Lambda(\Lambda[n])\cong R_\Lambda \Gamma_{\m{T}}^*\Lambda[n]\] the two definitions agree by the adjunction between $I_\Lambda$ and $R_\Lambda$.
Moreover, if $f\colon\m{T}\to \m{U}$ is a geometric morphism then for $x\in \m{T}$ we have 
\begin{align*}
&H^n(f_\sharp (x),\Lambda)=\Hom_{\m{U}}(f_\sharp x,\Gamma_{\m{U}}^*K(\Lambda,n))\cong  \Hom_{\m{T}}(x,f^*\Gamma_{\m{U}}^*K(\Lambda,n))\cong\\ &\cong \Hom_{\m{T}}(x,\Gamma_{\m{T}}^*K(\Lambda,n))=H^n(x,\Lambda)    
\end{align*}
so the functor $f_\sharp$ preserves the cohomology. 

%%%%%%%%%%%%%%%%%%%%%%%%%%%%%%%%%%%%%%%%%%
\subsection{Smooth Base Change for Higher Stacks}
%%%%%%%%%%%%%%%%%%%%%%%%%%%%%%%%%%%%%%%%%%

If $X$ is a scheme of finite type over $K$, then $h\Sh{X,\Lambda}$ is the derived category of sheaves of $\Lambda$-modules over $X$. In particular, every statement on $\Sh{X,\Lambda}$ that can be tested on the homotopy category can be approached using classical sheaf theory and $\etale$ cohomology. The most important for us is the smooth base-change theorem. 

\begin{proposition}[Smooth Base-Change for Schemes]
	\label{prop: smooth bc for schemes}
	Let $\Lambda$ be a finite ring of size prime to $char(K)$. 
	Let 
	\[
	\xymatrix{
		X\ar^{g'}[r]\ar^{f'}[d] & Y\ar^{f}[d] \\
		Z\ar^{g}[r] & W
	}
	\]
	be a pullback diagram of schemes of finite type over $K$, such that $f$ (and hence $f'$) is smooth. Then the diagram 
	\[
	\xymatrix{
		\Sh{X,\Lambda}\ar^{g'_*}[r]\ar^{f'_*}[d] & \Sh{Y,\Lambda}\ar^{f_*}[d] \\
		\Sh{Z,\Lambda}\ar^{g_*}[r] & \Sh{W,\Lambda}
	}
	\]
	satisfies the left BC-condition. Namely, the Beck-Chevalley map for this square, denoted $BC_*\colon f^*g_* \to g'_* (f')^*$, is an equivalence. 
\end{proposition}

\begin{remark}
	Note that this is the same as the right Beck-Chevalley condition for the square 
	\[
	\xymatrix{
		\Sh{X,\Lambda} & \Sh{Y,\Lambda} \ar^{{g'}^*}[l] \\
		\Sh{Z,\Lambda}\ar^{{f'}^*}[u] & \Sh{W,\Lambda}. \ar^{f^*}[u]\ar^{g^*}[l]
	}
	\]
\end{remark}

After prolongation, all the functors $f^*,g_*, g'_*, (f')^*$ admit left adjoints and $BC_*$ induces a natural transofrmation 
$BC_\sharp\colon f'_\sharp {g'}^*\to g^* f_\sharp$. This map is easily identified with the left Beck-Chevalley map of the square
\begin{equation}
\label{equation: square of smooth basechange diagram}
\xymatrix{
	\ProSh{X,\Lambda} & \ProSh{Y,\Lambda}\ar^{{g'}^*}[l] \\
	\ProSh{Z,\Lambda}\ar^{{f'}^*}[u] & \ProSh{W,\Lambda}\ar^{f^*}[u]\ar^{g^*}[l]
}
\end{equation}

As an immediate result, we get the following
\begin{corollary}
	\label{corollary: BC sharp equivalence for smooth of schemes}
	Let \[
	\xymatrix{
		X\ar^{g'}[r]\ar^{f'}[d] & Y\ar^{f}[d] \\
		Z\ar^{g}[r] & W
	}
	\]  
	be a pullback diagram of schemes of finite type over $K$ with $f$ smooth, and $\Lambda$ be a finite ring of order prime to the characteristic of $K$. Then the left Beck-Chevalley map 
	$BC_\sharp \colon f'_\sharp {g'}^*\to g^* f_\sharp$ 
	associated with the square (\ref{equation: square of smooth basechange diagram})
	is an equivalence. 
\end{corollary}

We shall now show that the conclusion of Corollary $\ref{corollary: BC sharp equivalence for smooth of schemes}$ holds for stacks as well.

\begin{definition}
	\label{def: smooth map of stacks}
	Let $f\colon\mathfrak{X}\to \mathfrak{Y}$ be a morphism of stacks. We say that $f$ is smooth if for every morphism $Y\to \mathfrak{Y}$ for a scheme $Y$ of finite type over $K$, the pull-back map $f'\colon Y\times_\mathfrak{Y}\mathfrak{X}\to Y$ is a smooth morphism of schemes. \end{definition}

\begin{theorem} [Smooth base-change for $\infty$-stacks]
	\label{theorem: Smooth Basechange for stacks}
	Let 
	\[
	\xymatrix{
		\mathfrak{X}\ar^{g'}[r]\ar^{f'}[d] & \mathfrak{Y}\ar^{f}[d] \\
		\mathfrak{Z}\ar^{g}[r] & \mathfrak{W}
	}
	\]   
	be a pullback square of stacks, with $f$ smooth and $g$ schematic. Let $\Lambda$ be a finite ring of order prime to $char(K)$. 
	Then the two squares  
	\[
	\xymatrix{
		\Sh{\mathfrak{X},\Lambda}\ar^{g'_*}[r]\ar^{f'_*}[d] & \Sh{\mathfrak{Y},\Lambda}\ar^{f_*}[d] \\
		\Sh{\mathfrak{Z},\Lambda}\ar^{g_*}[r] & \Sh{\mathfrak{W},\Lambda}
	}
	\]
	and 
	\[
	\xymatrix{
		\ProSh{\mathfrak{X},\Lambda} & \ProSh{\mathfrak{Y},\Lambda}\ar^{{g'}^*}[l] \\
		\ProSh{\mathfrak{Z},\Lambda}\ar^{{f'}^*}[u] & \ProSh{\mathfrak{W},\Lambda}\ar^{f^*}[u]\ar^{g^*}[l]
	}
	\]
	
	satisfy the left BC-condition. 
\end{theorem}

\begin{proof}
	We shall prove for the first square, as the second square follows from the first. 
	Let $I$ be a small infinity category and $W_\bullet \stackrel{\rho_\mathfrak{W}}\to \mathfrak{W}$ an $I^{\triangleright}$-shaped diagram exhibiting $\mathfrak{W}$ as a colimit of schemes of finite type. Pulling back this presentation to $\mathfrak{X}, \mathfrak{Y}$ and $\mathfrak{Z}$ along the morphisms in the pullback square, and using the fact that colimits in a topos are universal, we get a commutative cube of $I$-shaped diagrams in which all the faces are pulback squares and the front face consist of constant diagrams  
	\[
	\xymatrix{
		X_\bullet\ar^{g'_\bullet}[rr]\ar^{\rho^{\bullet}_{\mathfrak{X}}}[rd]\ar^{f'_\bullet}[dd] & & Y_\bullet \ar^{\rho^{\bullet}_{\mathfrak{Y}}}[rd]  \ar^(.6){f_\bullet}[dd]  & \\
		& \mathfrak{X} \ar^(.4){g'}[rr]\ar_(.3){f'}[dd] & & \mathfrak{Y}\ar^{f}[dd] & \\
		Z_\bullet\ar^(.6){g_\bullet}[rr]\ar^{\rho^{\bullet}_{\mathfrak{Z}}}[rd] & & W_\bullet \ar^{\rho^{\bullet}_{\mathfrak{W}}}[rd]&  \\ 
		& \mathfrak{Z}\ar^{g}[rr] & & \mathfrak{W}
	}
	\]
	
	Applying $\Sh{-,\Lambda}$ to the back-face we get a commutative square of $I^{op}$-shaped diagrams 
	\[\xymatrix{
		\Sh{X_\bullet, \Lambda} & \Sh{Y_\bullet, \Lambda} \ar^{{g'_\bullet}^*}[l] \\ 
		\Sh{Z_\bullet, \Lambda}\ar^{{f'_\bullet}^*}[u] & \Sh{W_\bullet, \Lambda}\ar^{f^*_\bullet}[u] \ar^{g^*_\bullet}[l] \\ 
	}
	\]
	with limit the square
	\[\xymatrix{
		\Sh{\mathfrak{X}, \Lambda} & \Sh{\mathfrak{Y}, \Lambda} \ar^{{g'}^*}[l] \\ 
		\Sh{\mathfrak{Z}, \Lambda}\ar^{{f'}^*}[u] & \Sh{\mathfrak{W}, \Lambda}\ar^{f^*}[u] \ar^{g^*}[l] \\ 
	}
	\]
	
	Hence, in view of Theorem \ref{theorem: inverse limits Beck Chevalley descent}, it would suffice to show that the maps $f_\bullet^*$ and ${f'_\bullet}^*$ satisfy the right Beck-Chevalley condition, and that for each $i\in I$ the square 
	
	\[\xymatrix{
		\Sh{X_i, \Lambda} & \Sh{Y_i, \Lambda} \ar^{{g'_i}^*}[l] \\ 
		\Sh{Z_i, \Lambda}\ar^{{f'_i}^*}[u] & \Sh{W_i, \Lambda}\ar^{f^*_i}[u] \ar^{g^*_i}[l] \\ 
	}
	\]
	satisfies the right Beck-Chevalley condition. Both the statements follows from the Smooth Base-change Theorem for schemes. 
\end{proof}

As an immediate consequence, we get
\begin{corollary}
	\label{corollary: relative homotopy smooth basechange}
	Let 
	\[
	\xymatrix{
		\mathfrak{X}\ar^{g'}[r]\ar^{f'}[d] & \mathfrak{Y}\ar^{f}[d] \\
		\mathfrak{Z}\ar^{g}[r] & \mathfrak{W}
	}
	\]   
	be a pullback square of stacks, with $f$ smooth and $g$ schematic. Let $\Lambda$ be a finite ring of order prime to $char(K)$. Then 
	\[
	g^*|f|_\Lambda  \cong |f'|_\Lambda
	\]
\end{corollary}

\section{Higher Obstruction Theory}
\label{RelativeObstruction}

Let $f\colon X\to Y$ be a morphism of schemes. If $f$ has a section, then the relative \etale 
topological type $\Et{(f)}$ has a global section. 
Thus, obstructions to global sections of $\Et{(f)}$ give rise to obstructions for sections of $f$. 

We will present obstruction theory in the context of higher topoi, 
starting with obstruction to a global section inside a given topos and
then extending it to the case of an obstruction for the existence of a section for a morphism of topoi.
We describe the formalism of obstruction theory using
the theory of gerbes,
this will make it easy to understand how the obstruction behaves via pullback.

\begin{remark}
	Throughout this section we only state theorems for dimensions $\ge 2$, 
	as the theory of gerbes requires some refinement to work for lower dimension.
	This is a due to the need for $\pi_n$ to be an abelian group. 
	
	We shall mension, however, that for the homological obstruactions the theory naturally extend to those cases as well.
\end{remark}

%%%%%%%%%%%%%%%%%%%%%%%%%%%%%
\subsection{Obstruction Theory for Global Sections of Sheaves in \texorpdfstring{$\infty$}{infinity}-topoi}
% % % % % % % % % % % % % % %
%%%%%%%%%%%%%%%%%%%%%%%%%%%%%

Let $\m{T}$ be an $\infty$-topos, and let $t$ be an object of $\m{T}$.
We ask whether $t$ has a global section, i.e. whether the terminal map 
$t\to *_\m{T}$ has a section.
Let's consider the Postnikov tower (as in \cite[def 5.5.6.23]{LurieHTT}) of $t$:
\[ 
t \to \Big[ \cdots\to P_{\le n}t\to\cdots\to P_{\le -1}t \Big]
\]
from which we see that if there is a global section for $t$ then 
there is a global section for each $P_{\le n}t$.
As in the usual construction in topological obstruction theory, 
we assume that we have found a global section $\sigma_{n-1}\colon  *\to P_{\le n-1}t$
and we will give an obstruction for the existence of a compatible global section $\sigma_n$ 
making the following diagram commutative:
\[
\xymatrix{
	&P_{\le n}t \ar[d]\\
	\ast \ar^{\sigma_n}@{-->}[ru] \ar^{\sigma_{n-1}}[r]  &P_{\le n-1}t.
}
\]

We do this as follows:
Let $\m{G}_n$ be the pullback of the above diagram. 
The existence of a lift as above is equivalent to the existence of a section
for the left downward arrow:
\[
\xymatrix{
	\m{G}_n \ar[r] \ar[d] &P_{\le n}t \ar[d] \\
	\ast \ar@/^/@{-->}[u] \ar[r] &P_{\le n-1}t.
}
\]

The pullback $\m{G}_n$ is an $n$-gerbe banded over $\pi_n t$,
as defined in \cite[7.2.2.20]{LurieHTT}.
The picture of $\m{G}_n$ to keep in mind is 
as a sheaf which is locally an Eilenberg-MacLane sheaf,
with homotopy group $\pi_n t$ at level $n$. 
Readers familiar with obstruction theory of topological spaces
should be quick to recognize the above generalization.

By \cite[7.2.2.28]{LurieHTT}, for $n\ge 2$ there is an equivalence between the set of 
equivalence classes of $n$-gerbes banded over a discrete abelian group object $A$
and the cohomology group $\HH^{n+1}_\m{T}(A)$. 
If $\m{G}$ is an $n$-gerbe banded by $A$,
we denote by $[\m{G}]$ it's corresponding cohomology class. 

The association of a gerbe to a cohomology class can be described as follows.
A class $\sigma\in\HH^{n+1}_\m{T}(A)$ is defined by a map 
\[
\ast_\m{T}\xrightarrow{\sigma} K(A,n+1). 
\]
The fiber of the above map is the corresponding $n$-gerbe.
Conversely, any gerbe can be fit uniquely as such a fiber
\[
\xymatrix{
	\m{G} \ar[r]\ar[d]          &   \ast_\m{T}     \ar^0[d] \\
	\ast_\m{T} \ar^(0.33)\sigma[r]    &   K(A,n+1). 
}
\]
\begin{definition}
	\label{cocycleObstruction}
	Let $n\ge 2$, let $t \in \m{T}$ be an element of an $\infty$-topoi, and let $\sigma_{n-1} \colon  *_\m{T} \to P_{n-1}(t)$ 
	be a section of the $n-1$-st Postnikov filtration of $t$. 
	Let $\m{G}_n$ denote the pull-back 
	$P_n \times_{P_{n-1}} *_\m{T}$. 
	Then we define the \textit{$n+1$-st obstruction class for extending $\sigma_{n-1}$}, denoted by $o_{n+1}(t,\sigma_{n-1})$, as the class $[\m{G}_n]$ in 
	$\HH^{n+1}_\m{T}(\pi_n(t))$.
\end{definition}

\begin{warning}
	Note the numbering conventions:
	\begin{itemize}
		\item The $n+1$-st homotopy obstruction corresponds to a $n+1$-th cohomology class.
		\item The $n$ gerbe is $n$-truncated and $n$-connective.
		\item The gerbe is a pullback over a global section of the $n-1$ postnikov truncation.
	\end{itemize}
\end{warning}

\begin{remark}
	This construction of obstruction theory is consistent with the classical obstruction theory 
	for lifting problems from classical topology. 
\end{remark}

\begin{remark}
	We shall always consider a gerbe banded by an abelian group object of the topos of discrete sheaves, $A\in Disc(\m{T})$, as 
	a pair $(\m{G},\phi)$ such that $\m{G}$ is an $n$-gerbe and 
	$\phi \colon  \pi_n(\m{G}) \stackrel{\sim}{\to} A$ is an isomorphism.
\end{remark}

%%%%%%%%%%%%%%%%%%%%%%%%%%%%%
\subsection{Functoriality Properties of the Obstructions}
% % % % % % % % % % % % % % %
%%%%%%%%%%%%%%%%%%%%%%%%%%%%%

In this sub-section we wish to show that the obstruction  
$o_{n+1}(t,\sigma_{n-1})$ is functorial in $t$, and also in $\m{T}$ w.r.t. geometric morphisms. 

Let $A,B$ be two abelian group objects in $Disc(\m{T})$. 
Let $f\colon  A \to B$ be a homomorphism. Then we get an induced map 
$f_* \colon  \HH^k_\m{T}(A) \to \HH^k_\m{T}(B)$.  
This is evident from the description of the cohomology as spaces of maps 
to Eilenberg-Maclane objects,
but we would like to describe this map in the language of gerbes. 

\begin{proposition}
	\label{functoriality of cohomology}
	Let $n\ge 2$.
	Let $(\m{G}, \phi)$ be an $n$-gerbe banded by $A$ and $(\m{G}',\phi')$ an $n$-gerbe banded by 
	$A'$. Let $f\colon  A \to A'$ be a map of group objects. Then 
	there is a map $g\colon  \m{G} \to \m{G}'$ inducing $f$ on $\pi_n$, if and only if 
	$f_*[\m{G}] = [\m{G}']$ in $\HH^{n+1}_\m{T}(A')$. 
\end{proposition}

\begin{proof}
	If $f_*[\m{G}] = [\m{G}']$ then we have a commutative diagram 
	
	\[
	\xymatrix{
		& \ar_{\sigma}[dl]\ast_\m{T} \ar^{\sigma'}[dr] & \\
		K(A,n+1)\ar^{K(f,n+1)}[rr] &            & K(A',n+1)
	}\]
	such that the fiber of $\sigma$ is $\m{G}$ and the fiber of $\sigma'$ is $\m{G}'$. 
	Then we have an induced map on the fibers which clearly induce the map $f$ on $\pi_n$, 
	by comparing the long exact sequences of homotopy groups associated to
	the morphism $K(f,n+1)$. 
	
	We shall prove the converse by passing to the associated complexes. 
	Let $C = P_{\le n}C_*(\m{G},\ZZ)$ and $C' = P_{\le n}C_*(\m{G}',\ZZ)$.
	The map $g\colon  \m{G} \to \m{G}'$ induces a map 
	$P_{\le n}C_*(g) \colon  C \to C'$, commuting with the augmentations. 
	In particular, we get a morphism of exact triangles 
	\[
	\xymatrix{
		K \ar[d]\ar[r]& C \ar[d]\ar[r] & \ZZ \ar@{=}[d]\\ 
		K' \ar[r]& C' \ar[r] & \ZZ
	}
	\]
	where $K$ and $K'$ are the fibers of the augmentations of $C$ and $C'$ respectively. 
	The banding maps $\phi \colon  \pi_n \m{G} \to A$ and $\phi' \colon  \pi_n \m{G}' \to A'$ 
	induces equivalences 
	$K \cong A[n]$ and $K' \cong A'[n]$, via the Hurewicz isomorphism.
	
	Shifting the exact triangles above and identifying $K$ and $K'$ with $A[n]$ and $A'[n]$ 
	we get a commutative square, who's commutativity follows from the assumption that 
	$g\colon  \m{G} \to \m{G}'$ induces the map $f$ on homotopy groups:
	\[
	\xymatrix{
		\ZZ \ar@{=}[d]\ar^{\alpha}[r] & A[n+1] \ar^{f[n+1]}[d] \\ 
		\ZZ \ar^{\alpha'}[r]          & A'[n+1].
	}
	\] 
	
	It remains to show that the morphisms $\alpha$ and $\alpha'$ 
	correspond via the functor $M$ to the morphisms 
	$\sigma \colon  \ast_\m{T} \to K(A,n+1)$ and $\sigma' \colon  \ast_\m{T} \to K(A,n+1)$ 
	classifying the cocycles $[\m{G}]$ and $[\m{G}']$ and thus 
	we have $\sigma' = f \circ \sigma$ so $f_* [\m{G}] = [\m{G}']$.
	The proof of this correspondence is postponed to the proof of Proposition \ref{EquivalenceOfHomologicalObstructions}.
	
\end{proof}

We now use the preceding proposition to prove the functoriality of the obstruction class with respect to a morphism in $\m{T}$.

Let $f\colon  t \to t'$ be a morphism in an $\infty$-topos $\m{T}$. 
Consider an $n-1$ section 
\[\sigma_{n-1}(t) \colon  \ast_T \to P_{\le n-1}(t).\] 
Composing with $P_{\le n-1}(f)$ we get a section 
\[P_{\le n-1}(f)_* \sigma_{n-1}(t) \colon  \ast_T \to P_{\le n-1}(t').\]
Let us denote 
$\sigma_{n-1}(t') = P_{\le n-1}(f)_* \sigma_{n-1}(t)$. 
Note that the fiber $\m{G}_n$ of the map $P_{\le n}t \to P_{\le n-1} t$ 
is an $n$-gerbe. 
We get a commutative diagram 
of Cartesian squares

\[
\xymatrix{
	\m{G}_n(t) \ar[rd] \ar[dd] \ar[rr]    &                        & P_{\le n}(t) \ar[dd] \ar[rd] &     \\
	&\m{G}_{n}(t')  \ar[rr] \ar[dd]             &              & P_{\le n}(t') \ar[dd]     \\
	\ast_\m{T}\ar[rd] \ar[rr]  &                       & P_{\le n-1}(t) \ar[rd] &               \\
	&   \ast_{\m{T}} \ar[rr]        &              & P_{\le n-1}(t')     \\                                                       
}
\]
and in particular we get a map $\m{G}_n(t) \to \m{G}_n(t')$ and we have a commuative diagram 
\[\xymatrix{
	\pi_n \m{G}_n(t) \ar^{\sim}[r] \ar[d] & \ar[d] \pi_n t \\ 
	\pi_n \m{G}_n(t') \ar^{\sim}[r]& \pi_n t'
}
\]

\begin{proposition} 
	Let $f\colon  t \to t'$ be a morphism in an $\infty$-topos $\mathcal{T}$. 
	Then 
	\[
	f_* o_{n+1}(t,\sigma_{n-1}) = o_{n+1}(t',P_{\le n-1}(f)(\sigma_{n-1}))
	\]
\end{proposition}

\begin{proof}
	This follows from Proposition \ref{functoriality of cohomology} and the existence of the map 
	$\m{G}_n(t) \to \m{G}_n(t')$. 
\end{proof}

Going on to explain the functoriality of the obstruction with respect to a geometric morphism,
we shall first say a word about
the map on cohomology induced by that geometric morphism. 
Indeed, if $f\colon  \m{T}' \to \m{T}$ is a geometric morphism, and 
$\phi\colon  \ast_\m{T} \to K(A,n+1)$ represent some element in $\HH^{n+1}_\m{T}(A)$ then we have 
$f^*K(A,n+1) \cong K(f^*A,n+1)$ and $f^* \ast_\m{T} \cong \ast_{\m{T}'}$ so 
we can regard $f^*\phi$ as representing a class in $\HH^{n+1}_\m{T}(f^*A)$. 
This construction gives us a pull-back map on cohomology associated with a 
geometric morphism. 

The functoriality of the obstruction itself can now be checked easily.
Since $f^*$ preserve finite limits, on the level of gerbes we obviously 
have $f^*[\m{G}] = [f^*\m{G}]$ (just compare the canonical fibrations of the corresponding 
Eilenberg-Maclane objects). 
Since $f^*$ also preserve Postnikov filtrations, we deduce the following:

\begin{proposition}
	$f^*o_{n+1}(t,\sigma_{n-1}) = o_{n+1}(f^*t,f^*\sigma_{n-1})$. 
\end{proposition}

%%%%%%%%%%%%%%%%%%%%%%%%%%%%%
\subsection{Obstruction Theory for a Geometric Morphism}
% % % % % % % % % % % % % % %
%%%%%%%%%%%%%%%%%%%%%%%%%%%%%

Let $f\colon  \m{T} \to \m{T}'$ be a morphism of $\infty$-topoi. 
The relative topological type for $f$ is defined as an object of 
of $\pro(\m{T}')$, given by $f_\sharp\ast_\m{T}$. 
It is easy to extend the definition of the obstruction for sections to pro-objects. 
Indeed, the postnikov tower of a pro-object is computed object-wise,  
so the fiber of the maps between consecutive filtra is a pro-gerbe, hence defines 
a cohomology class with coefficients in the corresponding pro-homotopy group
in the obvious way. Namely, if 
$\{\m{G}_i\}_{i \in I}$ is a system of gerbes banded over $\{A_i\}$ 
with compatibility maps between them, they corresponds 
to a compatible system of maps $\{\ast_\m{T} \to K(A_i,n+1)\}$ which is a pro-cohomology class 
by definition. 

\begin{definition}
	Let $f\colon  \m{T} \to \m{T}'$ be a geometric morphism of $\infty$-topoi. 
	Let $\sigma_{n-1} \colon  \ast_{\m{T}'} \to P_{\le n-1}(f_\sharp \ast_\m{T}).$
	Then we denote by  
	$o_{n+1}(f, \sigma_{n-1})$ the obstruction $o_{n+1}(f_\sharp \ast_\m{T}, \sigma_{n-1})$. 
\end{definition}

\begin{remark}
	The higher obstructions might depend on the chosen section, and the 
	possible options are governed by the differentials of the Bausfield-Kan spectral sequence associated 
	to the postnikov filtration of $t$.
	We will be mostly interested in $o_{n+1}(f,\sigma_{n-1})$ in cases where $f_\sharp \ast_{\m{T}}$ is 
	$n$-connected, and then the obstruction does not depend on the choice of section.
	In this case we shorten the notation and write $o_{n+1}(f)$, omitting the section 
	from the notation.
\end{remark}

\begin{proposition}
	As mentioned in the beginning of this section, if the morphism 
	$f\colon  \m{T} \to \m{T}'$ admits a section then all the obstructions vanish.
\end{proposition}
\begin{proof}
	Let $s\colon \m{T}'\to \m{T}$ be a section, so that $f\circ s=id_{\m{T}'}$.
	The terminal map $s_\sharp(\ast_{\m{T}'}) \to \ast_\m{T}$ gives a global section
	\[
	\ast_{\m{T}'} = id_\sharp(\ast_{\m{T}'}) = f_\sharp s_\sharp(\ast_{\m{T}'})\to f_\sharp(\ast_{\m{T}})
	\]
	and thus all the obstruction classes vanishes.
\end{proof}

%%%%%%%%%%%%%%%%%%%%%%%%%%%%%
\subsection{Homological Obstruction Theory and Extensions of Sheaves}
% % % % % % % % % % % % % % %
%%%%%%%%%%%%%%%%%%%%%%%%%%%%%
Let $\m{T}$ be an $\infty$-topos and $\Lambda$ be a ring.
Let $\Lambda_\m{T}=I_\Lambda(*_\m{T})$ denote the constant sheaf with value $\Lambda$ in $\Sh{\m{T},\Lambda}$.
Similarly to the homotopical obstruction, we can define the \textit{homological} obstruction to a section
\begin{definition}
	For $n\ge 2$, let $t$ be an object in $\m{T}$. Consider $t_\Lambda=R_\Lambda I_\Lambda t$, and let $\sigma_{n-1}:*_\m{T}\to t_\Lambda$. We define \textit{the $n+1$-st homological obstruction class for extending $\sigma_{n-1}$} as the homotopical obstruction class $o_{n+1}(t_\Lambda,\sigma_{n-1})$.
\end{definition}

It turns out that the lowest homological obstruction is given by an extension class in $\Ext_{\Sh{\m{T},\Lambda}}^{n+1}(\Lambda_\m{T}, \m{H}_n(\Et(f)))$,
which is the main claim of this subsection.

Let $H\in Disc(\Sh{\m{T},\Lambda})$ be a discrete sheaf of $\Lambda$ modules over $\m{T}$. 
Let $\alpha \in \Ext^{n+1}_{\Sh{\m{T},\Lambda}}(\Lambda_\m{T},H)$.
To $\alpha$ we can attach a fiber sequence 
\[H[n] \to M_\alpha \to \Lambda_\m{T},\] 
unique up to equivalence of fiber sequences. 

The map $R_\Lambda(M_\alpha) \to R_\Lambda(\Lambda_\m{T})$ is an equivalence after applying 
the Postnikov truncation $P_{n-1}$, 
as $R_\Lambda$ preserves fibre sequences and because $R_\Lambda H[n]$ is $n-1$-connected. 
On the other hand, $P_{n-1}(R_\Lambda(\Lambda_\m{T}))$    
has a canonical section corresponding to the element $1 \in \Lambda$ hence inducing a section $\sigma_{n-1} \colon \ast_\m{T} \to  P_{n-1}(R_\Lambda(M_\alpha))$. 
Thus, we can define the obstruction class 
\[o_n^{\alpha} \coloneqq o_n(R_\Lambda(M_\alpha),\sigma_{n-1}) \in \HH^{n+1}_\m{T}( \pi_n(R_\Lambda(M_\alpha)))=\HH_\m{T}^{n+1}(H).\]

\begin{proposition}
	\label{EquivalenceOfHomologicalObstructions}
	In the notation above, for $n\ge 2$, the canonical isomorphism 
	\begin{equation}
	\label{equivOfHAndExt}
	\HH^{n+1}_{\m{T}}(H) \cong \Ext^{n+1}_{\Sh{\m{T},\Lambda}}(\Lambda_\m{T},H)
	\end{equation}
	sends $o_\alpha$ to $\alpha$. 
\end{proposition}

The proof is based on the following observation. 

\begin{lemma}
	\label{cocycles adjunction}
	Let $f: \Lambda \to H[n+1]$ be a map, classified by cohomology class
	$\beta \in \Ext^{n+1}_{\mtl}(\Lambda, H)$. 
	Then the composition 
	\[*_\m{T} \xrightarrow{1_\Lambda} R_\Lambda(\Lambda_\mt) \xrightarrow{R_\Lambda(f)} R_\Lambda(H[{n+1}]) \cong K(H,{n+1})\]
	is classified by $\beta$, under the equivalence \eqref{equivOfHAndExt}.
\end{lemma}
\begin{proof}
	The identification \eqref{equivOfHAndExt} of $\pi_0(\Hom_\m{T}(*_\m{T}, K(H,n+1)))$ with $\Ext^{n+1}_{\mtl}(\Lambda_\m{T},H)$, is 
	given by the adjunction isomorphism
	\[\Hom_{\mt}(*_\m{T},R_\Lambda(H[{n+1}])) \cong \Hom_{\mtl}(\Lambda \otimes *_\m{T},H[{n+1}])\]
	using the identifications $R_\Lambda(H[{n+1}]) \cong K(H,{n+1})$ 
	and $\Lambda \otimes \ast_\m{T} = \Lambda_\mt$. 
	The unit map for the above adjunction can be computed on the terminal object directly
	to be $\ast_\m{T} \xrightarrow{1_\Lambda} R_\Lambda(\Lambda \otimes \ast_\m{T}) = R_\Lambda(\Lambda_\mt)$.
	The composition 
	\[*_\m{T} \xrightarrow{1_\Lambda} R_\Lambda(\Lambda) \xrightarrow{R_\Lambda(f)}  K(H,{n+1})\] 
	can therefore be identified with the composition 
	\[*_\m{T} \xrightarrow{1_\Lambda} R_\Lambda(\Lambda \otimes *_\m{T}) 
	\xrightarrow{R_\Lambda(f)}R_\Lambda(H[{n+1}]),\]
	which is exactly the morphism corresponding to 
	\[f\in\Hom_{\Sh{\m{T},\Lambda}}(\Lambda\otimes\ast_\mt,H[n+1])\]
	via the adjunction $I_\Lambda \dashv R_\Lambda$, and thus to the class $\beta$.
\end{proof}

\begin{proof}[Proof of Proposition \ref{EquivalenceOfHomologicalObstructions}]
	
	First of all, the fiber sequence  
	\[H[n] \to M_\alpha \to \Lambda_\mt\] 
	induces by shifting a fiber sequence  
	\[M_\alpha \to \Lambda_\mt \to H[n+1].\] 
	The homotopy class of the map $\Lambda\mt \to H[n+1]$ above is classified by 
	$\alpha$. 
	
	Applying the functor $R_\Lambda$ to this fiber sequence we get 
	a pullback diagram 
	\[
	\xymatrix{
		R_\Lambda(M_\alpha) \ar[r]\ar[d]   &K(\Lambda, 0) \ar[d]\\
		\ast_\m{T} \ar[r]                  &K(H, n+1)
	}
	\]
	
	Note that the map $*_\m{T} \to K(H, n+1)$ in the diagram is the canonical base-point, i.e. classified by the trivial cocycle, since it is the $R_\Lambda$ of the zero map $0 \to H[n+1]$.  
	
	Consider now the map $1_\Lambda : *_\m{T} \to R_\Lambda(\Lambda_\mt)$, picking the point $1$. 
	We can extend the pullback square above to the diagram 
	\[
	\xymatrix{
		\m{G}_\alpha \ar[r] \ar[d]          & \ast_\m{T} \ar^{1_\Lambda}[d] \\
		R_\Lambda(M_\alpha) \ar[r]\ar[d]   &K(\Lambda,0) \ar[d]\\
		\ast_\m{T} \ar[r]               &K(H, n+1)
	}
	\]
	in which both of the squares are pullback squares. 
	By the definition of the cocycle $o_n^\alpha$, it remains to show that the composition 
	\[
	*_\m{T} \xrightarrow{1_\Lambda} K(\Lambda,0) \to K(H, n+1)
	\]
	is classified by the cocycle $\alpha$. Indeed, in this case 
	$\m{G}_\alpha$, which is the gerbe classified by $o_n^\alpha$, is the equalizer of the 
	canonical section of $K(H,n+1)$ and the section classified by $\alpha$, 
	hence classified by $\alpha$ by \cite[Proposition 7.2.2.8]{ LurieHTT}. 
	This last result follows directly from Lemma \ref{cocycles adjunction}.
\end{proof}

\section{Quadratic Bundles}
\label{FormsAndMaps}

In this section we calculate and recall the properties of the universal quadratic bundle 
that are needed for the final computation of the obstruction. 
The notation and definitions will be the same as in \cite{jardine1992cohomological}.

%%%%%%%%%%%%%%%%%%%%%%%%%%%%%
\subsection{The Universal Quadratic Bundle}
% % % % % % % % % % % % % % %
%%%%%%%%%%%%%%%%%%%%%%%%%%%%%
Let $BG$ denote the simplicial variety defined over a field $K$ of characteristic$\ne 2$
given by the standard  
Milnor realization, namely $BG_k = G^k$ with the usual face and degeneracy maps. 
It is known that $BG$ classifies \etale-locally trivial principle $G$ bundles.
In our case, for $G=O_n$, we claim that $BG$ presents the stack classifying quadratic bundles.

\begin{definition}
	Let $X$ be a scheme over $K$. A \textbf{quadratic bundle of rank $n$}  over $X$ is a 
	locally free sheaf $\Epsilon$ over $X$ of rank $n$ together with a non-degenerate pairing 
	$B\colon  \Epsilon \otimes \Epsilon \to \mathcal{O}_X$ of sheaves of $\mathcal{O}_X$ modules. 
\end{definition}

%\begin{remark}
%Note that the category of trivial quadratic bundles of the form $(\mathcal{O}^n_X,\sum x_i^2)$ is equivalent 
%to the category of trivial principle $O_n$ bundles. 
%This is because for a quadratic bundle we can associate the bundle of orthogonal bases with the natural $O_n$ action.
%Conversely, given a trivial principal $O_n$ bundle we can take the span of the columns with the obvious quadratic form.
%\end{remark}
A quadratic bundle of rank $d$ is called $\emph{trivial}$ if it is isomorphic to the quadratic bundle $(\m{O}_X^d,\sum_i x_i^2)$.
The stack of locally trivial quadratic bundle is represented by the simplicial scheme $BO_n$. To see this, note that there is n equivalence between the groupoid of locally trivial quadratic bundles on a scheme $S$ and principal $BO_n$-bundles on $S$. Indeed, let $V_n$ be the standard quadratic representation of $BO_n$. If $P$ is a principal $O_n$ bundle then $P \times_{O_n} V$ is a quadratic bundle. Conversely, if $(\m{E},B)$ is a quadratic bundle, then the collection of orthonormal frames of $\m{E}$ is a principal $O_n$-bundle, locally trivial if $\m{E}$ is. These two constructions give an equivalence of the functors on schemes given by locally trivial quadratic bundles and $BO_n$.     

\begin{lemma}
	\label{lemma: quadratic bundles are locally trivial}
	Every quadratic bundle is locally trivial in the $\etale$ topology.
\end{lemma}

\begin{proof}
	Let $(\Epsilon,B)$ be a quadratic bundle over $X$. Then Zariski locally, $\Epsilon$ is a trivial vector bundle. 
	Thus, we may assume without lost of generality that $\Epsilon \cong \mathcal{O}_X^n$ and $B$ is some symmetric matrix 
	of functions, which is everywhere non-degenerate. 
	Choose a closed point $x \in X$. The reduction to $k(x)$ of $B$ can be diagonalized. Choose 
	a matrix $\bar{A}$ such that $\bar{A} B|_{k(x)} \bar{A}^T$ is diagonal, and lift $\bar{A}$ arbitrarily to a neighborhood of 
	$x$. We may assume that $A$ is everywhere defined, by shrinking $X$. 
	Then, since $B_{1,1}$ is invertible at $x$ and by shrinking $X$ further, we may assume that $B_{1,1}$ is invertible. 
	By applying row and column operations we can now eliminate all the other entries in the first row and column of $B$. 
	
	Applying the same procedure to the first minor of $B$, and inductively to all matrices 
	$\{B_{i,j}\}_{i,j \ge k}$ we can find some Zariski Neighborhood $X'$ and invertible matrix $A'$ of functions on 
	$X'$ such that $A' B (A')^T$ is diagonal, say $A'B(A')^T = diag(f_1,...,f_n)$. 
	Hence, we can assume that $B$ itself is the diagonal matrix $diag(f_1,...,f_n)$.
	Since $B$ is non-degenerate, the $f_i$-s are invertible. 
	Since $char(K) \ne 2$, we can add $\sqrt{f_i}$ to $\m{O}_XX$ to get an \etale cover of $X$. 
	Formally, we consider $Y = \{(x,y_1,...,y_n) \in X \times \AA^n_K \colon  y_i^2 = f_i(x)\}$ 
	and the projection $p_{X}\colon  Y \to X$ is an \etale cover. Moreover, 
	$p_{X}^*(B) = diag(y_1^2,...,y_n^2) \equiv diag(1,..,1)$ 
	via the matrix $diag(1/y_1,...,1/y_n)$. 
\end{proof}

\begin{corollary}
	\label{lemma: BO_n classifies quadratic bundles}
	$BO_n$ is the stack classifying quadratic bundles.
\end{corollary}

\begin{corollary}
	$\Hom(\spec(K), BO_n)$ is equivalent to the groupoid of non-degenerate quadratic forms of rank $n$ over $K$. 
\end{corollary}

Let $S^n$ denote the standard $n$-sphere over $K$. If $B$ is a quadratic form, 
let $S_B$ denote the unit sphere of $B$, given by  
\[S_B = \{v \in \AA^{n+1} \colon  B(v,v) = 1\}.\] 
Then $S^n = S_{x_0^2 + \cdots + x_n^2}$. 

A similar argument to the proof of Lemma \ref{lemma: quadratic bundles are locally trivial} shows that the quotient stack
$O_{n+1} \quo S^n$ classifies quadratic bundles with an orthonormal section.
The groupoid of such data, $(\Epsilon,B,\sigma)$, is canonically equivalent to the groupoid of quadratic bundles of rank $n$, via 
the rank $n$ quadratic bundle $(\Epsilon,B,\sigma)\mapsto(\sigma^\perp,B|_{\sigma^\perp})$. 
Thus, we have an equivalence
\[O_{n+1} \quo S^n \cong BO_n.\]

To distinguish the quotient stacks from the objects they classify, we denote by
$Q_{n+1}$ the stack of quadratic rank $n+1$ bundles,
and by $Q_{n+1}^S \cong Q_{n}$ the stack of 
quadratic bundles with a length 1 section. 
They are indeed stacks, as they are equivalent to stacks represented by $BO_n$,
as displayed in the following commutative diagram 
\begin{equation}
\label{equivalenceOfQuadraticThingies}
\xymatrix{
	Q_{n} \ar^{\sim}[r] \ar[rd] & \ar[d] \ar^{\sim}[r]Q_{n+1}^S & \ar[d] \ar^{\sim}[r] O_{n+1} \quo S^{n} & \ar[dl] BO_{n} \\ 
	& \ar^\sim[r] Q_{n+1}                 & BO_{n+1}  
}
\end{equation}

%\begin{proposition} 
%The morphism $Q^S_n \to Q_n$ is a local fibration. 
%\end{proposition}

%\begin{proof}
%Since both can be considered as sheaves of 1-groupoids, one has to check only lifting properties for morphisms. 
%However, the morphism $Q_n^S \to Q_n$ is a left fibration because we can pull back sections of norm 1. Thus, if we have 
%an arrow $f \colon  \Epsilon \to \mathcal{F}$ in $Q_n$ and a section of $\mathcal{F}$ (or $\Epsilon$), 
%we can pull it back to $\Epsilon$ (or pull back to $\mathcal{F}$ along $f^{-1}$), to get a section of $\mathcal{E}$ (or $\mathcal{F}$). 
%These are exactly the lifting properties required from a fibration.
%\end{proof}

\begin{proposition}
	Let $X \stackrel{f_B}{\to} Q_n$ be a morphism classifying a quadratic bundle $(\Epsilon,B)$. 
	Then the pullback 
	$f_B^* Q_n^S$ is the sphere bundle of $(\Epsilon,B)$ over $X$.  
	In particular, the morphism $Q_n^S \to Q_n$ is schematic.
\end{proposition}

%\begin{proof}
%Let $S_B(\Epsilon)$ denote the sphere bundle of $(\Epsilon,B)$.
%Since the morphism $p_n \colon  Q_n^S \to Q_n$ is a local fibration, the homotopy pullback along $f_B$ is just the 
%naive one. Unwinding the definition, we see that its sections over a scheme $g\colon  Y \to X$ over $X$ are exactly length 1 
%sections of the pullback of $(\Epsilon,B)$ to $Y$, namely sections of the map $S_{g^*B}(g^*\Epsilon) \to Y$. 
%This is exactly the stack represented by the projection $S_B(\Epsilon) \to X$. 
%\end{proof}
\begin{proof}
	A morphism $g:Y\to S_B(\m{E})$ determines a morphism $\bar{g}:Y\to S_B(\m{E}) \to X$. Since $\m{E}$ admits a canonical unit section over $S_B(\m{E})$, its pullback along $g$ determines a unit section of $\m{E}$. Conversely, if $\bar{g}:Y\to X$ is a morphism and $\sigma$ is a unit section of $\bar{g}^*\m{E}$, then we get a map $g:Y\to S_B(\m{E})$, since pulback of vector bundles commutes with formation of total space in a way compatible with unit sections. 
	
	Hence, $S_B(\m{E})$ classifies pairs of a morphism to $X$ and a unit section of the pulback of $\m{E}$. This is by definition the groupoid classified by the pulback $X\times_{Q_n}Q_n^S$. 
\end{proof}

\begin{corollary}
	Let $B$ be a quadratic form over $K$, then we have a pullback diagram 
	\[\xymatrix{
		S_B \ar[r] \ar[d]& Q_n^S \ar[d] \\
		spec(K) \ar[r]  & Q_n 
	}
	\]
\end{corollary}

%%%%%%%%%%%%%%%%%%%%%%%%%%%%%
\subsection{Stiefel-Whitney Classes and the Cohomology of \texorpdfstring{$BO_n$}{BOn}}
% % % % % % % % % % % % % % %
%%%%%%%%%%%%%%%%%%%%%%%%%%%%%

Here we recall the cohomological properties of $BO_n$ without proofs, 
as computed in \cite{jardine1989universal}.

The \etale cohomology of $BO_n$ with $\Z2$ coefficients is the tensor product of the geometric and the arithmetic cohomologies.
\begin{theorem}[{\cite[Thm. 1]{jardine1989universal}}]
	\label{CohomologyOfBOn}
	Let $K$ be a field of characteristic $\neq 2$. And let $A$ denote the mod 2 Galois cohomology ring 
	$\HH^*_{\et}(K,\Z2)$ of $K$.
	Then there is an isomorphism of graded algebras of the form 
	\[
	\HH^*_{\et}(BO_n, \Z2) \cong A[HW_1,\ldots, HW_n]
	\]
	where the polynomial generator $HW_i$ has degree $i$.
\end{theorem}
\begin{remark}
	The classes $HW_i$ get their notation from "Hasse-Witt" invariants, which are the arithmetic analogue of the Stiefel Whitney classes, associated with billinear forms.
\end{remark}

Furthermore, Whitney's product formula holds. 
In this paper we will only need the special case of the product formula associated with 
the diagonal embedding $\Delta\colon  BO_1^n\to BO_n$, which is a consequence of the proof of the above theorem.
Using the Kunneth formula, it is easy to show that 
\[
\HH^*_{\et}(BO_1^n, \Z2) = A[HW_1^{(0)},\ldots,HW_1^{(n-1)}]
\]
where $HW_1^{(i)}$ are all in degree 1, the pullbacks of $HW_1$ under the projections $BO_1^n\to BO_1$,
and $A$ is the mod 2 cohomology of $K$, as above.
\begin{theorem}[{\cite[Thm. 2.13]{jardine1989universal}}]
	In the notations above, 
	\[
	\Delta^*(HW_{i}) = \sigma_i(HW^{(0)}_1,\ldots, HW^{(n-1)}_1),
	\]
	where $\sigma_i$ is the $i$'th symmetric polynomial in $n$ variables. 
\end{theorem}
\begin{corollary}
	\label{WhitneyProductFormula}
	In particular, $\Delta^*(HW_{n}) = HW^{(0)}_1\cup \cdots \cup HW^{(n-1)}_1$. 
\end{corollary}

The cohomology groups behave in the expected way for the canonical morphism $BO_m\to BO_n$.
\begin{theorem}
	Let $m<n$, and consider the map $BO_m\to BO_n$, given by the standard embedding of the orthogonal groups. 
	Then the corresponding map of cohomologies is given by
	\begin{align*}
	A[HW_1,\ldots,HW_n]     &\to    A[HW_1,\ldots,HW_m]\\
	\forall i\le m.\ HW_i   &\to    HW_i    \\
	\forall i>m.\ HW_i      &\to    0.
	\end{align*}
\end{theorem}
\begin{proof}
	Consider the commutative diagram
	\[
	\xymatrix{
		\HH^*_\et(BO_n,\Z2)\ar@{^{(}->}[r]\ar[d]   &    \HH^*_\et(BO_1^n,\Z2) \ar[d]\\
		\HH^*_\et(BO_m,\Z2)\ar@{^{(}->}[r]         &    \HH^*_\et(BO_1^m,\Z2). 
	}
	\]
	The right downward arrow, defined by the embedding of $BO_1^m$ to the first $m$ factors
	of $BO_1^n$, is the morphism of polynomial rings 
	\[A[x_1,\ldots,x_n]\to A[x_1,\ldots,x_m]\]
	given by sending $x_i$ to $x_i$ if $i\le m$ or to 0 otherwise.
	It remains to observe that this projection sends $\sigma_i(x_1,\ldots,x_n)$ to $\sigma_i(x_1,\ldots,x_m)$ if $i\le m$ and to 0 otherwise.
\end{proof}

%%%%%%%%%%%%%%%%%%%%%%%%%%%%%
\subsection{The Relative Homological Type of a Sphere Bundle}
% % % % % % % % % % % % % % %
%%%%%%%%%%%%%%%%%%%%%%%%%%%%%

The aim of this subsection is to calculate the relative \etale homological type for the projection $p_n\colon  Q_n^S \to Q_n$ of stacks over $K$.  

We will show the following arithmetic version of classical obstruction theoretic interpretation of the Stiefel-Whitney classes: 

\begin{theorem}
	\label{push of universal sphere}
	Let $p_n \colon  Q_n^S \to Q_n$ denote the natural projection. Then the sheaf of $\mathbb{Z}/2$ modules
	$\Z2 \otimes \Et(p_n)$ is the extension of $\Z2$ by $\Z2[n-1]$ classified by the Hasse-Witt class 
	\begin{align*}
	HW_n\in \HH^n_\et(Q_n,\Z2) &\cong \Hom_{Sh_\infty((Q_n)_\et,\Z2)}(\Z2,\Z2\left[n\right]) \cong \\ &\cong\Ext^1_{Sh_\infty((Q_n)_\et,\Z2)}(\Z2,\Z2[n-1]).
	\end{align*}
\end{theorem}

The proof of the first part of this result, 
that $\Z2\otimes\Et(p_n)$ is an extension as above, 
is done by a d\v{e}vissage argument starting from the homotopy type of the standard sphere, 
going through a trivial bundle over a scheme, 
and then proven for $Q_n$ by writing it as a colimit of schemes with a trivial quadratic bundle.

We start by analyzing the case of the trivial sphere bundle over the point. 

\begin{proposition}
	\label{homotopy type of the sphere}
	The fibre of the terminal map from the relative homological realization over $K$ of the standard sphere $\Et_{/K}(S^{n-1})$
	to $\Z2$ is
	\[
	\Z2[n-1] \to \Et_{/K}(\Z2\otimes S^{n-1}) \to \Z2.
	\]
\end{proposition}

\begin{proof}
	Consider the fibre $\mathcal{G}$ of the terminal map $\Et_{/K}(\Z2\otimes S^{n-1}) \to \Z2$. 
	We shall show that $\mathcal{G}$ is isomorphic to $\Z2[n-1]$. 
	Since the group $\Z2$ has no non trivial automorphisms, it suffices to check that $\mathcal{G}\cong\Z2[n-1]$ after base change to the separable closure $\overline{K}$.
	This result will follow if we can show that the \etale homology of $\overline{\mathcal{G}}$ is concentrated in degree $n-1$, and equal to $\Z2$ at this degree.
	
	For the calculation of the \etale homology groups over $\overline{K}$, it is enough to
	calculate the \etale cohomology groups of the sphere $\overline{S^{n-1}}$, by the universal coefficient theorem.
	In \cite[Exp. XII]{SGA7-II} it is proven that the reduced \etale cohomology groups of the spheres with coefficients in $\ZZ_2$ consists of one copy of $\Z2$ concentrated in degree $n-1$, and thus it also follows for \etale  homology with coefficients in $\Z2$ by the universal coefficients theorem.
\end{proof}

The next case is for trivial bundles over $K$-schemes. 
\begin{proposition}
	\label{proposition: case of trivial bundle}
	Let $X$ be a scheme over $K$, and let $(\mathcal{E},B)$ be a trivial quadratic bundle of rank $n$ on $X$. 
	Let $p_B \colon  S_B(\Epsilon) \to X$ denote the projection from the sphere bundle of $B$ to the 
	base $X$. Then the fiber of the natural map $\Z2 \otimes \Et(p_B) \to \Z2$ is equivalent to the constant sheaf 
	$\Z2[n-1]$. 
\end{proposition}
\begin{proof}
	In this case $S_B(\Epsilon) \cong X \times S^{n-1}$ as a scheme over $X$. 
	Thus, the problem reduces to Proposition \ref{homotopy type of the sphere} via the 
	smooth base change theorem \ref{prop: smooth bc for schemes}, as we now explain. 
	Consider the diagram
	\[
	\xymatrix{
		X\times S^{n-1} \ar^g[r]  \ar^q[d]  &   S^{n-1} \ar^p[d]  \\
		X               \ar^f[r]          &   \ast            .
	}
	\]
	
	Note that the map $p$ is smooth, as $\text{char}(K)\ne 2$, and thus the smooth base change theorem holds in this case. 
	
	By Corollary \ref{corollary: relative homotopy smooth basechange} we get
	\[
	\Et(q) \otimes \mathbb{Z}/2 \cong f^*\Et(p) \otimes \mathbb{Z}/2.
	\]
	
	This concludes the reduction, as $f^*$ preserves fibre sequences and sends constant sheaves to constant sheaves.
\end{proof}

Finally, we are ready to consider the relative homological type of the universal sphere bundle.
First, we show that the sheaf $\Z2 \otimes \Et{(p_n)}$ is indeed an extension 
of $\Z2[n-1]$ by $\Z2$. 
\begin{proposition}
	\label{case of Qn}
	Let $p_n \colon  Q_n^S \to Q_n$ denote the natural projection. Then 
	$\Z2 \otimes \Et(p_n)$ is an extension of $\Z2$ by $\Z2[n-1]$. 
\end{proposition}
\begin{proof} 
	Let $O_n^{\bullet} :\Delta^{op} \to Sch_{/K}$  denote the simplicial diagram giving the Milnor realization of $BO_n$. Let $p_n^\bullet: O_{n-1}^\bullet \to O_n^{\bullet}$ denote the standard inclusions. Then 
	$p_n: Q_n^S \to Q_n$ is isomorphic to the colimit $\colim p_n^\bullet$.
	
	Let $\rho_\bullet: O_n^\bullet \to Q_n$ and $\rho^S_\bullet: O_{n-1}^\bullet \to Q_n^S$  denote the comparison maps. 
	Let $\tilde{\rho}^* \colon \Sh{(Q_n)_\et} \adj \oplaxlim \Sh{O_n^\bullet} \colon\tilde{\rho}_\sharp$ be the induced map on sheaves (as in the proof of Theorem \ref{theorem: formula for push from colimit}). Similarly, for $\rho^S$ let  $(\tilde{\rho}^S)^*$ and $\tilde{\rho}^S_\sharp$ be the corresponding adjoint functors. 
	
	It follows from Theorem \ref{theorem: formula for push from colimit} that the natural map 
	$\mu:\tilde{\rho}_{\sharp} (p_n)_{lax,\sharp}  (\tilde{\rho}^S)^* \to (p_n)_{\sharp}$ is an equivalence at $\ast_{Q_n^S}$, namely 
	\[
	\mu\colon \tilde{\rho}_{\sharp}  (p_n)_{lax,\sharp} (\tilde{\rho}^S)^* \ast_{Q_n^S} \stackrel{\sim}{\to} (p_n)_{\sharp} \ast_{Q_n^S} = \Et(Q_n^S).
	\]
	
	The object $\tilde{\rho}^* \ast_{Q_n}$ is final in $\oplaxlim \pro\Sh{O_{n}^\bullet}$ and hence there is a unique map 
	$Aug: (p_n)_{lax,\sharp} (\tilde{\rho}^S)^* \ast_{Q_n^S} \to \tilde{\rho}^* \ast_{Q_n}$ in $\oplaxlim \pro Sh_\infty(O_n^\bullet)$ 
	which restrict to the terminal map $\Et(p_n^{[k]})\to \ast_{O_n^{[k]}}$ at every $[k] \in \Delta^{op}$. 
	Tersoring $Aug$ with $\mathbb{\ZZ}/2$ we obtain a map 
	\[
	Aug_{\mathbb{Z}/2}:(p_n)_{lax,\sharp} (\tilde{\rho}^S)^* \mathbb{Z}/2\cong \mathbb{Z}/2 \otimes (p_n)_{lax,\sharp} (\tilde{\rho}^S)^* \ast_{Q_n^S} \stackrel{Aug \otimes \mathbb{Z}/2}{\to} \mathbb{Z}/2 \otimes \tilde{\rho}^* \ast_{Q_n} \cong \tilde{\rho}^* \mathbb{Z}/2  
	\]
	where the first and last equivalences are the ones in  \ref{corollary: shrieck commutes with tensor}.
	Let $F$ denote the fiber of $Aug_{\mathbb{Z}/2}$, so we have a fiber sequence 
	\[
	\xymatrix{
		F\ar[r] & (p_n)_{lax,\sharp} (\tilde{\rho}^S)^* \mathbb{Z}/2 \ar^(.6){Aug_{\mathbb{Z}/2}}[r]&   \tilde{\rho}^* \mathbb{Z}/2\\
	}
	.\]
	Applying $\tilde{\rho}_\sharp$ to this sequence we get a fiber sequence 
	\[
	\xymatrix{
		\tilde{\rho}_{\sharp} F\ar[r] & \tilde{\rho}_{\sharp} (p_n)_{lax,\sharp} (\tilde{\rho}^S)^* \mathbb{Z}/2 \ar^(.6){Aug_{\mathbb{Z}/2}}[r]&   \tilde{\rho}_{\sharp} \tilde{\rho}^* \mathbb{Z}/2\\
	}
	.\]
	
	The result will now follow from this fiber sequence once we show that 
	$\tilde{\rho}_{\sharp} \tilde{\rho}^* \mathbb{Z}/2 \cong \mathbb{Z}/2$ in $\pro \Sh{Q_n,\mathbb{Z}/2}$ and that 
	$\tilde{\rho}_\sharp F \cong \mathbb{Z}/2[n-1]$. 
	
	For the first, let $id: Q_n \to Q_n$ denote the identity map, we get from Theorem \ref{theorem: formula for push from colimit} that 
	\[\mu: id = id_\sharp \to \tilde{\rho}_{\sharp} id_\sharp \tilde{\rho}^* \cong \tilde{\rho}_\sharp \tilde{\rho}^*\] is an equivalence at $\ast_{Q_n}$ and hence it is also an equivalence after taking the tensor product with $\mathbb{Z}/2$. 
	
	We pass the the second equivalence $\tilde{\rho}_\sharp F \cong \mathbb{Z}/2[n-1]$.
	Note that by the Smooth Base-change Theorem (\ref{theorem: Smooth Basechange for stacks}), after tensoring with $\mathbb{Z}/2$ the object $(p_n)_{lax,\sharp} (\tilde{\rho}^S_{lax})^* \ast_{Q_n^S}$ lies in the limit.
	In other words, by Theorem \ref{theorem: BC map then has left adjoint}, this object lies in the full-subcategory $\invlim \pro\Sh{O_n^\bullet,\mathbb{Z}/2}$ of $\oplaxlim \pro\Sh{O_n^\bullet,\mathbb{Z}/2}$. 
	Since the same is true for $\tilde{\rho}^* \mathbb{Z}/2$, we see that $F$ lies also in $\invlim \pro\Sh{O_n^\bullet,\mathbb{Z}/2}$.
	Moreover, by Proposition \ref{proposition: case of trivial bundle}, the value of $F$ at every $O_n^{[k]}$ is equivalent to $\mathbb{Z}/2[n-1]$. 
	In particular, $F[1-n]$ has value $\mathbb{Z}/2$ at every object $O_n^{[k]}$,
	and hence it is equivalent to $\tilde{\rho}^*\mathbb{Z}/2$ since $\mathbb{Z}/2$ has no nontrivial automorphisms. Hence 
	\[\tilde{\rho}_{\sharp} F \cong \tilde{\rho}_{\sharp} \tilde{\rho}^*\mathbb{Z}/2[n-1] \cong \mathbb{Z}/2[n-1],\] using again the equivalence $\tilde{\rho}_{\sharp} \tilde{\rho}^* \cong id$ at objects of $\Sh{(Q_n)_\et}$.    
\end{proof}

To conclude the proof of Theorem \ref{push of universal sphere}, 
we will see that this extension is classified by $HW_n$. 
Consider the fiber sequence 
\begin{equation}
\label{fiberSequenceForC}
\Z2[n-1] \to \Z2\otimes\Et(p_n) \to \Z2.
\end{equation}
Recall that sheaf cohomology with coefficients in $\Z2$ is computed by applying $\Hom_{Sh_\infty((Q_n)_\et,\Z2)}(-,\Z2)$, abbreviated by $\Hom_{Q_n}(-,\Z2)$, on the abelian sheaf
and taking the corresponding homotopy group ($\Z2$-modules) of the resulting Hom-spaces.
Applying the above to the fiber sequence \ref{fiberSequenceForC}, we get a long exact sequence of the form 
\begin{align}
\label{longExactSequenceForQn}
\cdots &\to \pi_{n-1}(\Hom_{Q_n}(\Z2[n-1],\Z2)) \to \pi_n(\Hom_{Q_n}(\Z2,\Z2)) \to\\
&\to \pi_n(\Hom_{Q_n}(\Z2 \otimes \Et{(p_n)},\Z2)) \to \cdots \nonumber
\end{align}
However, by adjunction we have 
\[
\Hom_{Q_n}(\Z2 \otimes \Et{(p_n)},\Z2) \cong \Hom_{Q_n^S}(\Z2,\Z2) 
\]
and the resulting map 
\begin{align*}
\HH_\et^n(Q_n,\Z2) &= \pi_n(\Hom_{Q_n}(\Z2,\Z2)) \to\\
&\to \pi_n(\Hom_{Q_n}(\Z2 \otimes \Et{(p_n)},\Z2))\cong\\
&\cong \pi_n(\Hom_{Q_n^S}(\Z2,\Z2)) \cong \\
&\cong \HH_\et^n(Q_n^S,\Z2)
\end{align*}
corresponds to the pullback map induced on cohomology by $p_n$. 

Recall that 
\[
\HH_\et^*(Q_n,\Z2) \cong A[HW_1,\ldots,HW_n]
\]
while 
\[
\HH_\et^*(Q_{n}^S,\Z2) \cong \HH_\et^*(Q_{n-1},\Z2) \cong A[HW_1,\ldots,HW_{n-1}].
\]
Moreover, the induced map on cohomology, $p_n^*$, is 
given by reduction modulo the ideal generated by $HW_n$. 
In particular, the kernel of $p_n^*$ in degree $n$ is
one dimensional and generated by $HW_{n}$ as a subspace. 

Also, we clearly have 
\[\pi_{n-1}(\Hom_{Q_n}(\Z2[n-1],\Z2)) \cong \pi_0(\Hom_{Q_n}(\Z2,\Z2)) \cong \Z2\]
since $Q_n$ is connected. 

Thus, the portion on the long exact sequence \eqref{longExactSequenceForQn} presented above is isomorphic to 
\[
\ldots \to \Z2 \stackrel{\delta}{\to} (A[HW_1,\ldots,HW_n])^n \xrightarrow{\text{mod }HW_n}
(A[HW_1,\ldots,HW_{n-1}])^n \to \ldots 
\]
and exactness of the sequence now forces the map $\delta$ to be given by $\delta(1) = HW_n$.  
It follows that the extension \eqref{fiberSequenceForC} is classified by $HW_n$, proving Theorem \ref{push of universal sphere}.

\section{Obstructions for Arithmetic Spheres}
\label{ObstructionsForSpheres}
Carrying on to compute the obstruction classes, we shall start with dimensions 0,1.
These obstructions are computed directly using simple computations interesting in their own right, and can be read independently of the rest of this paper.

After that the computation for dimensions $\ge 2$ is carried on using the full machinery developed in the sections above.
It is interesting to note that the theorems proved earlier in this paper could be 
extended to allow for a proof of low dimensions, 
as the theory of gerbes behaves properly in low dimensions for \emph{abelian} sheaves and we are dealing with the homological obstruction.

%%%%%%%%%%%%%%%%%%%%%%%%%%%%%
\subsection{Obstruction Theory for Sphere Bundles of Dimensions 0,1}
% % % % % % % % % % % % % % %
%%%%%%%%%%%%%%%%%%%%%%%%%%%%%
The computations below only rely on the definitions given in \cite{HarpazSchlank}. 
Rigour will be diminished for the benefit of clarity and brevity. 
\begin{proposition}[Case of 0-dimensional sphere]
	Let $a\in K^\times$.
	The first \etale $\Z2$-homology obstruction for the existence of a rational point on $ax^2=1$
	is precisely the class $[a]$.
\end{proposition}
\begin{proof}
	As in \cite{HarpazSchlank}, the relative \etale homotopy type is given by taking the absolute \etale
	homotopy type after base change to an algebraic closure, and remembering the Galois action.
	Then, the homotopy obstruction is obtained by the usual obstruction for the existence of a fixed point
	for the Galois action.
	
	In our case, the \etale homotopy type is equivalent to 2 points with a Galois action
	similar to the action on $\{\sqrt{a},-\sqrt{a}\}$. 
	Thus, the homotopy obstruction is given by arbitrarily selecting a point and considering the
	map
	\begin{align*}
	\Gamma_K    &\to        \{\sqrt{a},-\sqrt{a}\} \\
	\sigma      &\mapsto    \sigma \sqrt{a},
	\end{align*}
	which is indeed  an element in $\HH^1(\spec K, \{\sqrt{a},-\sqrt{a}\})$.
	Therefore, after passing to the homology obstruction, 
	we get exactly the morphism $\Gamma_K\to\Z2$ corresponding to the class $[a]$.
\end{proof}

\begin{proposition}[Case of 1-dimensional sphere]
	Let $a,b\in K^\times$.
	The second \etale $\Z2$-homology obstruction for the existence of a rational point on 
	the sphere $X$ defined by the quadratic equation $ax^2+by^2=1$
	is precisely the class $[a]\cup [b]$.
\end{proposition}
\begin{proof}
	In \cite{HarpazSchlank}, the relative \etale homotopy type of $X$ over $K$ is defined by the inverse system 
	over all hypercovers $U_\bullet \to X$ of the component simplicial sets 
	$\pi_0 \overline{U}_\bullet$, obtained by base-change to the separable closure and taking the $\Gamma_K$ set of connected components levelwise. 
	In \cite{BARNEA2016784}, the relationship between this definition of the relative \etale homotopy type and the one given in this paper is discussed. 
	Homotopical questions will get the same answer in both cases, and in particular the obstruction class over a field can be computed by choosing appropriate hypercovers. 
	
	Let $U_\bullet\to X$ be a hypercover and denote by $\overline{U_\bullet}$ its base-change to the separable closure.
	Considering the Galois action on the set of connected components, we have a $\Gamma_K$-equivariant map 
	$\Et_{/K}(X) \to \pi_0(\overline{U_\bullet}) =\colon |{U_\bullet}|$, and in particular it induces a 
	map $\pi_n \Et_{/K}(X) \to \pi_n |{U_\bullet}|$. 
	Moreover, the induced map 
	\[
	\HH^{n+1}(\Gamma_K,\pi_n \Et_{/K}(X)) \to \HH^{n+1}(\Gamma_K,\pi_n |{U_\bullet}|)
	\]
	maps the obstructions 
	for rational point of $X$, to the obstructions for homotopy fixed point in 
	$|{U_\bullet}|$. For the special case $n=1$, the map $\pi_1 \Et_{/K}(X) \to \pi_1 |{U_\bullet}|$ 
	should be interpreted as morphism of fundamental groupoids. 
	
	We will pick a specific hypercover $U_\bullet$ and compute the fundamental groupoid of $\overline{U_\bullet}$ together with the $\Gamma_K$-action, 
	and show that it is a $\ZZ/2\ZZ$-gerbe, in the sense that it is a groupoid 
	with $\Gamma_K$-action, which is non-equivariantly equivalent to $B\ZZ/2\ZZ$. 
	Hence, by the simple observation that the morphism $\HH_1(\Et_{/K}(X),\ZZ/2\ZZ) \to \HH_1(|U_\bullet|,\ZZ/2\ZZ)$ is 
	an isomorphism, 
	we see that we have an isomorphism
	\[
	\HH^{2}(\Gamma_K,\pi_1(\Z2\otimes\Et_{/K}(X)))\cong \HH^2(\Gamma_k,\pi_1|U_\bullet|)\cong \HH^2(\Gamma_K,\Z2).
	\]
	It remains to show that this gerbe is classified by the cocycle $[a]\cup[b]$.
	Consider the field extension $L=K[\sqrt{a},\sqrt{-b}]$ of $K$.
	We'll assume throughout the remainder of this proof that this extension is of degree 4, as well as assuming $i=\sqrt{-1}\in K$,
	both assumptions are not needed for the conclusion to hold but simplify the proof while the proofs for the general case only require small adjustments from the proof below and are left to the interested reader.
	This extension is Galois, with Galois group $\Gamma_{L/K}=\Z2\oplus\Z2$. 
	We shall build the hypercover with respect to this field, and the obstruction class will 
	be seen to reside in $\HH^2(\Gamma_{L/K},\Z2)\subset \HH^2(\Gamma_K,\Z2)$.
	
	Observe that $X$ is a twist of $\mathbb{G}_m$, trivialized over $L$.
	We construct a degree 8 cover $U\to X$ by taking the Weil restriction from $L$ to $K$ of the double cover 
	\[
	\xymatrix{
		\mathbb{G}_m \ar[r]^{z\mapsto z^2} &\mathbb{G}_m.
	}
	\]
	Concretely, we can write
	\[
	\xymatrix{
		(z,s,t) \ar@{|->}[d]  &   U=\spec K[z,s,t]/(s^2-a,t^2+b) \ar[d]\\
		(\frac{z^2+z^{-2}}{2s},\frac{z^2-z^{-2}}{2t})   &   X=\spec K[x,y]/(ax^2 + by^2 -1).
	}
	\]
	
	This cover can be extended to a hypercover $U_\bullet\to X$ by taking the \'{C}ech nerve,
	so that $U_\bullet$ is a simplicial variety given by 
	\[
	U_n = \overbrace{U\times_X\cdots\times_X U}^{n\text{ times}} =
	\spec\frac{K[z_1,z_1^{-1},\ldots,z_n,z_n^{-1},s_1,t_1\ldots,s_n,t_n]}{(s_i^2-a,t_i^2+b,z_i^2-(\frac{s_i}{s_j} + \frac{t_i}{t_j})\frac{z_j^2}{2} - (\frac{s_i}{s_j} - \frac{t_i}{t_j})\frac{z_j^{-2}}{2})}.
	\]
	
	The 0 dimensional skeleton of $|U_\bullet|$ consists of 4 points,
	each point is indexed by one solution of $s,t$ to $s^2=a$ and $t^2 = -b$.
	The Galois group $\Gamma_{L/K}$ acts on the set of points naturally.
	
	The edges of $|U_\bullet|$ (the connected components of $\overline{U\times_X U}=\overline{U}\times_{\overline{X}}\overline{U}$) are grouped into pairs,
	so that any two vertices have two edges among them for each direction.
	For two points indexed by $s_1,t_1$ and $s_2,t_2$, the two morphisms from the former to the latter are given by the two solution to 
	\[
	z_1^2=(\frac{s_1}{s_2} + \frac{t_1}{t_2})\frac{z_2^2}{2} + (\frac{s_1}{s_2} - \frac{t_1}{t_2})\frac{z_2^{-2}}{2},
	\]
	so that
	\[
	z_1 \in \left\{ \pm z_2^{\pm 1}, \pm i z_2^{\pm 1}  \right\}.
	\]
	For example, if $s_2 = -s_1$ and $t_2 = t_1$ then we have the two morphisms indexed by $z_1 = \pm i z_2^{-1}$, which we think of as the morphisms $z\mapsto \pm i z^{-1}$.
	The Galois group acts on the edges by taking morphisms to morphisms with the same "name" (the same choice of solution $z_1=\pm \sqrt{\pm 1}z_2^{\pm 1}$).
	Note that the Galois group does indeed preserve $\frac{s_1}{s_2}=\frac{\sigma s_1}{\sigma s_2}$. 
	
	The 2 dimensional faces are thought of as homotopies between the morphisms, 
	so that there is a homotopy iff the diagram is commutative,
	as can be computed directly. 
	From all of the above,
	it is easy to calculate that $\pi_1(|U_\bullet|)=\Z2$.
	
	Now we can carry on to calculate the obstruction class. 
	In general, if $\Gamma$ is a group acting on a topological space $X$, we can use the 
	skeleton filtration of $E\Gamma$, instead of the Postnikov filtration of 
	$X$, in order to compute the obstructions. 
	In the case of the first obstruction class of a connected space $X$, the formula takes the following form. 
	Pick any $x \in X$, and for each $\sigma \in \Gamma$, choose a path $\gamma_\sigma$ connecting 
	$x$ to $\sigma(x)$. Then, for each simplex $[1,\sigma,\tau] \in E\Gamma_2$, we get 
	a map $\partial \Delta^2 \to X$ 
	given by the following triangle in the space $X$: 
	
	\[
	\xymatrix{
		&  x \ar^{\gamma_\sigma}[ld] \ar^{\gamma_\tau}[rd]        &  \\
		\sigma(x) \ar^{\sigma(\gamma_{\sigma^{-1}\tau})}[rr]&          & \tau(x)
	}
	\]
	
	This triangle represents a class in $\pi_1(X,x)$, so we get a map 
	$E\Gamma_2 \to X$ which is easily seen to be a 2-cocycle. This cocycle is exactly the obstruction 
	to a homotopy fixed point of $\Gamma$ in $X$, or in other words, the obstruction for section 
	to the map $X // \Gamma \to B\Gamma$.  
	
	This cocycle is indeed $[a]\cup [b]$, as we can explicitly calculate.
	We exemplify this calculation on one value of the cocycle.
	Let $\sigma_a,\sigma_b\in\Gamma_{L/K}$ be the automorphisms satisfying
	\begin{align*}
	\sigma_a(\sqrt{a})=-\sqrt{a},\ \sigma_a(\sqrt{-b})=\sqrt{-b}, \\
	\sigma_b(\sqrt{-b})=-\sqrt{-b},\ \sigma_b(\sqrt{a})=\sqrt{a}. \\
	\end{align*}
	If we pick the base point defined by $s=\sqrt{a},t=\sqrt{-b}$ 
	and edges
	\begin{align*}
	&\gamma_1=z,\ &\gamma_{\sigma_a}=iz^{-1}\\
	&\gamma_{\sigma_b}=z^{-1},\ &\gamma_{\sigma_a \sigma_b}=iz,
	\end{align*}
	then we can calculate the cocycle on $\sigma_a,\sigma_b$ as 
	\begin{align*}
	\gamma_{\sigma_a}^{-1}\circ 
	\sigma_a(\gamma_{\sigma_b}^{-1})\circ
	\gamma_{\sigma_a\sigma_b} 
	= (iz^{-1})\circ(z^{-1})\circ(iz) =-z
	\end{align*}
	which is not the chosen map $\gamma_1$, and hence results in a nontrivial element.
	Carrying on this calculations to other pairs, this identifies the cocycle with the cup product.
\end{proof}

%%%%%%%%%%%%%%%%%%%%%%%%%%%%%
\subsection{Obstruction Theory for Sphere Bundles of Higher Dimensions}
% % % % % % % % % % % % % % %
%%%%%%%%%%%%%%%%%%%%%%%%%%%%%

In this subsection we finally relate the Stiefel-Whitney classes to the 
relative obstruction theory of sphere bundles. 
By Proposition \ref{EquivalenceOfHomologicalObstructions} and the fact that $\Z2 \otimes \Et{(\pi_n)}$ is an extension of $\Z2$ by 
$\Z2[n-1]$ classified by $HW_n$ (Theorem \ref{push of universal sphere}), we deduce that 
the $n$-th homological obstruction for section of $\pi_n \colon  Q_n^S \to Q_n$ is exactly 
$HW_n$. 
Using this and a base-change result, we will show that for every morphism 
$f\colon  X \to Q_n$ classified by a quadratic bundle 
$(\mathcal{E},B)$ over $X$, the extension 
$\Z2[n-1] \to \Z2 \otimes \Et{(\pi_B)} \to \Z2$ is classified by 
the pullback $f^*(HW_n)$.  

\begin{proposition}
	Consider the pullback diagram 
	\[\xymatrix{
		S_B(\Epsilon) \ar^{f'}[r] \ar^{\pi_B}[d] &  Q_n^S \ar^{\pi_n}[d] \\ 
		X \ar^{f}[r]             &  Q_n
	}
	\] 
	
	The canonical comparison map
	\[\Z2 \otimes \Et{(\pi_B)} \to \Z2 \otimes f^* \Et{(\pi_n)}  \]
	is an isomorphism. 
\end{proposition}
\begin{proof}
	The map $Q_n^S\to Q_n$ is a smooth map of stacks, in the sense of Definition \ref{def: smooth map of stacks},
	as it is a schematic map presented by a simplicial map $O_{n-1}^\bullet\to O_n^\bullet$ which is level-wise smooth. Hence, by theorem \ref{theorem: Smooth Basechange for stacks}, the diagram 
	\[
	\xymatrix{
		\ProSh{S_B(\Epsilon),\Z2}       &   \ProSh{Q_n^S,\Z2} \ar^(0.45){(f')^*}[l] \\
		\ProSh{X,\Z2}   \ar^{\pi_B}[u]  &   \ProSh{Q_n,\Z2} \ar^{f^*}[l]    \ar^{\pi_n^*}[u]
	}
	\]
	satisfies the left BC-condition, which precisely means that the map above is in fact an isomorphism.
\end{proof}

\begin{theorem}
	\label{obstruction of sphere bundles}
	Let $(\Epsilon,B)$ be a quadratic bundle of rank $n$ over a scheme $X$ over $K$ 
	classified by a map $f_{\Epsilon,B}\colon  X \to Q_n$. 
	Let $\pi_B \colon  S_B(\Epsilon) \to X$ denote the projection of the sphere bundle to the base $X$.
	Then, the mod $2$ obstruction for section of $\pi_B$ is the $n$-th Stiefel Whitney class of 
	$(\Epsilon,B)$, 
	namely 
	\[o_n(\pi_B,\Z2) = f_{\Epsilon,B}^* HW_n\] 
\end{theorem}

\begin{proof}
	Recall that the obstruction is compatible with morphisms of sheaves and with 
	pullbacks by morphisms of schemes. 
	Since 
	$\Z2 \otimes \Et{\pi_B} \cong f_{\Epsilon,B}^* \Z2 \otimes \Et{\pi_n}$, 
	we get 
	\[
	o_n(\Z2 \otimes \Et{\pi_B}) = f_{\Epsilon,B}^* o_n(\Z2 \otimes \Et{\pi_n}) = f_{\Epsilon,B}^*(HW_n).
	\]
\end{proof}

%%%%%%%%%%%%%%%%%%%%%%%%%%%%%
\subsection{Application to Unit Spheres}
% % % % % % % % % % % % % % %
%%%%%%%%%%%%%%%%%%%%%%%%%%%%%

Let $B$ be a quadratic form of an ${n+1}$ dimensional vector space over a field $K$ of characteristic $\ne 2$. 
We can diagonalize $B$ to a form of the form 
$a_0x_0^2 + \ldots + a_n x_n^2$. 
The theory developed in Section \ref{FormsAndMaps} implies a map
$\spec{K} \stackrel{f_B}{\to} Q_{n+1}$, and by Theorem
\ref{obstruction of sphere bundles}, 
the obstruction for a rational point of $S_B$ is given by 
$f_B^* HW_{n+1}$.

Moreover, we can write $V$ as a product of 1 dimensional quadratic spaces
\[
V = V_0\times \cdots \times V_n
\]
where each space $V_i$ is equipped with the billinear form $B_i(x)=a_i x^2$.
Each $V_i$ is classified by a map $\spec{K}\stackrel{f_{B_i}}{\to}Q_1$, 
and thus we can factor $f_B$ as
\[
\spec{K} \xrightarrow{f_{B_0}\times\cdots\times f_{B_n}} Q_1^{n+1} \stackrel{\Delta}{\to} Q_{n+1}.
\]

Recall that, for a field $K$, $\HH^1(K,\Z2) \cong K^\times / (K^\times)^2$. 
We denote by $[a]$ the cohomology class corresponding to 
$a \in K^\times$. 

\begin{lemma}
	\label{pullbackOfHW1}
	Let $X$ be the 0 dimensional sphere given by $ax^2=1$,
	and let $f\colon \spec K \to Q_1$ be the map classifying $X$.
	We have
	\[
	f^*(HW_1)=[a].
	\]
\end{lemma}
\begin{proof}
	Note that $Q_1\cong B\Z2$ and that $Q_1^S\cong E\Z2 \cong \spec{K}$.
	
	It is well known that 1-cocycles classify Galois covers. 
	In our case, the class identified with $HW_1$ in $\HH^1(B\Z2,\Z2)$ classifies the cover 
	$E\Z2\to B\Z2$.
	Therefore, $f^*(HW_1)$ classifies the cover obtained via pullback, $X\to \spec K$.
	This cocycle can be directly computed to be $[a]$.
\end{proof}

\begin{proposition}
	If $B$ is a quadratic form with diagonal entries 
	$a_0,\ldots,a_n$, then $f_B^* HW_n = [a_0] \cup \ldots \cup [a_n]$. 
\end{proposition}

\begin{proof}
	The factorization $f_B=\Delta\circ (f_{B_0}\times\cdots\times f_{B_n})$ gives
	\begin{align}
	f_B^* HW_n  &=  (f_{B_0}^*\times\cdots\times f_{B_n}^*)\circ \Delta^*HW_n = \label{fromWhitney}\\
	&=  (f_{B_0}^*\times\cdots\times f_{B_n}^*) (HW_1^{(0)}\cup\cdots\cup HW_1^{(n)})\label{fromCupProp}=\\
	&=  f_{B_0}^*HW_1^{(0)}\cup\cdots\cup f_{B_n}^*HW_1^{(n)} \label{from0dimComp}=\\
	&=  [a_0]\cup\cdots\cup[a_n]. \notag
	\end{align}
	The equality \eqref{fromWhitney} is true due to the Whitney formula \eqref{WhitneyProductFormula}.
	Equality \eqref{fromCupProp} is true by general properties of the cup product.
	The last equality, \eqref{from0dimComp}, follows from Lemma \ref{pullbackOfHW1}.
\end{proof}

\begin{corollary}
	Let $B$ be the quadratic form defined over $K$ given by
	\[
	B(x_0,\ldots,x_n) = \sum_i a_i x_i^2.
	\]
	The $n$-th mod 2 obstruction for a rational point on the sphere $S_B$ is 
	\[HW_n(B) = [a_0] \cup\ldots\cup[a_n]\]
\end{corollary}
\begin{proof}
	This follows from the previous proposition by applying Theorem \ref{obstruction of sphere bundles}.
\end{proof}

\bibliographystyle{unsrt}
\bibliography{ref}

\end{document}